\documentclass[11 pt]{article}  
\usepackage[utf8]{inputenc}
\usepackage{amsmath}
\usepackage{amsfonts}
\usepackage{amssymb}
\usepackage{graphicx}
\usepackage{mathrsfs}
\usepackage{upref,amsthm,amsxtra,exscale}
\usepackage{cite}
\usepackage[colorlinks=true,urlcolor=blue,
citecolor=red,linkcolor=blue,linktocpage,pdfpagelabels,
bookmarksnumbered,bookmarksopen]{hyperref}

\usepackage{fullpage}

\usepackage{subcaption}
\usepackage{caption}
\usepackage{cleveref}
\usepackage{comment}

\newtheorem{theorem}{Theorem}[section]

\newtheorem{lemma}[theorem]{Lemma}
\newtheorem{proposition}[theorem]{Proposition}

\numberwithin{equation}{section}
\theoremstyle{definition}\newtheorem{remark}[theorem]{Remark}

\usepackage{enumitem}

\def\N{\mathbb{N}}

\def\cI{\mathcal{I}}

\def\cN{\mathcal{N}}

\def\R{\mathbb{R}}

\def\cL{\mathcal{L}}
\def\mH{\mathbb{H}}
\def\cE{\mathcal{E}}

\def\cH{\mathcal{H}}
\def\weakto{\rightharpoonup}

\def\sideremark#1{\ifvmode\leavevmode\fi\vadjust{\vbox to0pt{\vss
 \hbox to 0pt{\hskip\hsize\hskip1em
 \vbox{\hsize2.1cm\tiny\raggedright\pretolerance10000
  \noindent #1\hfill}\hss}\vbox to15pt{\vfil}\vss}}}%

    \usepackage{color}
\usepackage[dvipsnames]{xcolor}

\newcommand{\weakly}{\rightharpoonup}

\begin{document}

\title{
Small order asymptotics for nonlinear fractional problems
}
\author{V\'ictor Hern\'andez-Santamar\'ia\thanks{Instituto de Matem\'aticas, Universidad Nacional Aut\'onoma de M\'exico, Circuito Exterior, C.U., C.P. 04510 CDMX, Mexico. E-mails: \texttt{victor.santamaria@im.unam.mx}, \ \texttt{alberto.saldana@im.unam.mx}} \and Alberto Salda\~{n}a\footnotemark[1]}

\date{}

\maketitle
\begin{abstract}

We study the limiting behavior of solutions to boundary value nonlinear problems involving the fractional Laplacian of order $2s$ when the parameter $s$ tends to zero. In particular, we show that least-energy solutions converge (up to a subsequence) to a nontrivial nonnegative least-energy solution of a limiting problem in terms of the logarithmic Laplacian, \emph{i.e.} the pseudodifferential operator with Fourier symbol $\ln(|\xi|^2)$. These results are motivated by some applications of nonlocal models where a small value for the parameter $s$ yields the optimal choice. Our approach is based on variational methods, uniform energy-derived estimates, and the use of a new logarithmic-type Sobolev inequality.
\end{abstract}

\noindent{\bf 2020 MSC} (Primary) 
35B40; 
35S15, 
35J60, 
35R11 

{\small
\noindent{\bf Keywords:}  Nehari manifold, small order expansion, logarithmic nonlinearity, fractional Lane-Emden equation, mountain pass
}

\section{Introduction}

Nonlocal operators provide an important tool to model phenomena with anomalous diffusive behavior. Examples of this can be found in a wide variety of situations: in water waves, crystal dislocations, phase transitions, peridynamics, and finance, see \emph{e.g.} \cite{BV16,HB10,L04}. A particularly illustrative example of how the nonlocality can play a prominent role in a model comes from population dynamics \cite{PV18}, where a nonlinear (logistic-type) equation in terms of a fractional Laplacian of order $2s$ is used to describe the movement of a species. Here, a small order $s$ describes an almost static population (but capable of moving quickly long distances), whereas $s$ near $1$ relates to a very dynamic species which mostly moves short distances. Therefore, the parameter $s$ accounts for different dispersal strategies. Interestingly, in \cite{PV18}, it is shown that a very small order $s$ can be the best strategy for survival if the habitat is not too fragmented or not too hostile in average. 

Similarly, a small value for the exponent $s$ yields the optimal choice in other applications such as optimal control \cite{SV17}, approximation of fractional harmonic maps \cite{ABS21}, and fractional image denoising \cite{AB17}. This motivates the research of small order asymptotics of nonlinear fractional problems. Moreover, the mathematical structure that arises in this analysis is rich and interesting in itself from a theoretical point of view. 

In this paper, we develop a method to study the small order asymptotics of subcritical nonlinear problems. To make the main ideas in our arguments more transparent, we focus on a model problem with power-type nonlinearity, and we refer to Remark \ref{ext:rmk} for a discussion on extensions and generalizations. To be more precise, we consider the following fractional Dirichlet problem
\begin{align}\label{subcritical:intro}
 (-\Delta)^s u_s = |u_s|^{p_s-2}u_s\quad \text{ in }\Omega,\qquad u_s=0\quad \text{ in }\R^N\backslash\Omega,
\end{align}
where $N\geq 1$, $s\in(0,\frac{1}{4})$, $p_s\in(2,2^*_s)$ is superlinear and subcritical, $2^*_s=\frac{2N}{N-2s}$ is the critical Sobolev exponent, and $\Omega\subset \R^N$ is a bounded open Lipschitz set. Here $(-\Delta)^s$ denotes the integral fractional Laplacian given by the hypersingular integral in the principal value sense
\begin{align*}
(-\Delta)^s u(x)
= c_{N,s}\, p.v.\int_{\R^N}\frac{u(x)-u(y)}{|x-y|^{N+2s}}\ dy,\qquad  c_{N,s}:= 4^s\pi^{-\frac{N}{2}}s(1-s)\frac{\Gamma(\tfrac{N}{2}+s)}{\Gamma(2-s)}.
\end{align*}

\medskip

Our goal is to answer the following question:
\begin{align}\label{q}
    \text{\emph{What is the limit of a solution $u_s$ of~\eqref{subcritical:intro} as $s\to 0^+$?} }
\end{align}

Note that (at least formally) both sides of equation \eqref{subcritical:intro} behave similarly in the limit, namely, $(-\Delta)^s u\to u$ and $|u|^{p_s-2}u\to u$ as $s\to 0^+$. Therefore, to answer question~\eqref{q}, it is necessary to consider the first order expansion (with respect to $s$) on both sides of equation~\eqref{subcritical:intro}. This leads us naturally to the logarithmic Laplacian $L_\Delta$, which was introduced in \cite{CW19} and has the following pointwise evaluation
\begin{align*}
L_\Delta u(x) =  c_N\, p.v.\int_{B_1(x)}\frac{u(x)-u(y)}{|x-y|^{N}}\ dy
-c_N\int_{\R^N\backslash B_1(x)}\frac{u(y)}{|x-y|^N}\ dy + \rho_N\, u(x),
\end{align*}
where
$c_N$ and $\rho_N$ are explicit constants (see~\eqref{cN:def} and~\eqref{rho:def}) and
$B_1(x)$ is the open ball in $\R^N$ of radius 1 centered at $x$. Furthermore, as shown in \cite{CW19}, $L_\Delta$ is a pseudodifferential operator with Fourier symbol $2\ln(|\xi|)$ and, for $u\in C^\infty_c(\R^N)$,
\begin{align*}
 L_\Delta u = \lim_{s\to 0^+}\frac{d}{ds}(-\Delta)^s u.
 \end{align*}
 
We give the following answer to question~\eqref{q}. We refer to Section \ref{not:sec} for the precise definitions of least-energy solutions and related notation. In particular, $\cH^{s_k}_0(\Omega)$ and $\mH(\Omega)$ denote suitable Hilbert spaces (see \eqref{Hs:def} and \eqref{Hdef}) and $J_{s_k}:\cH^{s_k}_0(\Omega)\to\R$ and $J_0:\mH(\Omega)\to\R$ are the \emph{energy functionals} associated to~\eqref{equsk} and~\eqref{log:prob:intro} (see \eqref{Js:def} and \eqref{J0def}).

\begin{theorem}\label{main:thm}Let $N\geq 1$ and let $\Omega\subset \R^N$ be a bounded open Lipschitz set. For $s\in(0,\frac{1}{4})$, let $p_s:=p(s)$, where 
 \begin{align}\label{p:hyp}
 p\in C^1([0,\tfrac{1}{4}]),\quad 
2<p(s)<2^*_s:=\frac{2N}{N-2s}\ \ \text{ for $s\in\left(0,\tfrac{1}{4}\right)$,}\quad \text{ and }
\quad p'(0)\not\in\left\{0,\tfrac{4}{N}\right\}.    
\end{align}
 Let $(s_k)_{k\in\N}\subset(0,\frac{1}{4})$ be such that $\lim_{k\to\infty}s_k=0$ and, for $k\in\N,$ let $u_{s_k}\in\cH^{s_k}_0(\Omega)$ be a least-energy solution of
 \begin{align}\label{equsk}
 (-\Delta)^{s_k} u_{s_k} = |u_{s_k}|^{p_{s_k}-2}u_{s_k}\quad \text{ in }\Omega,\qquad u_{s_k}=0\quad \text{ in }\R^N\backslash\Omega.
 \end{align}
 Then there is a least-energy solution $u_0\in \mH(\Omega)\backslash\{0\}$ of 
 \begin{align}\label{log:prob:intro}
  L_\Delta u_0 =\mu\ln(|u_0|)u_0\ \ \text{ in }\Omega,\qquad u_0=0\ \ \text{ in }\R^N\backslash \Omega,
 \qquad  \mu
 :=p'(0),
 \end{align}
 such that, passing to a subsequence, 
 \begin{align}\label{cL2}
 u_{s_k}\to u_0\quad \text{ in $L^2(\R^N)$ as $k\to\infty$.}
 \end{align}
  Moreover,
\begin{align}\label{energies}
\lim_{k\to\infty}\frac{1}{s_k}J_{s_k}(u_{s_k})=J_0(u_0)>0\qquad \text{ and }\qquad 
\lim_{k\to\infty}\|u_{s_k}\|_{s_k} = |u_0|_2.
\end{align}
\end{theorem}
 Note that~\eqref{p:hyp} includes the paradigmatic case 
\begin{align*}
p_s=p(s):=\lambda\, 2^*_s+(1-\lambda)\, 2\qquad \text{ for some $\lambda\in(0,1)$.}     
\end{align*}
The condition $p_s<2^*_s$ is needed to guarantee the existence of solutions (for example, if $p_s=2^*_s$, the solvability of \eqref{subcritical:intro} depends on the domain $\Omega$; see \emph{e.g.} \cite{HSS21} for some existence and nonexistence results for \eqref{subcritical:intro} with $p_s=2^*_s$ for any $s>0$). An interesting phenomenon is that, for~\eqref{log:prob:intro}, the ``logarithmic subcriticality" is reflected on the scalar factor $\mu$, where the fact that $\mu=p'(0)<\frac{4}{N}$ is crucial in our arguments to obtain compactness. See Remark \ref{open:rmk} for a discussion on the case $p'(0)\in\{0,\frac{4}{N}\}$.

One of the main obstacles to show Theorem~\ref{main:thm} is to find a uniform bound in $\mH(\Omega)$ for the sequence of solutions $(u_{s_k})_{k\in\N}$. To show this, we use an ``intermediate inequality" (see Lemma~\ref{prebd}) between the fractional Sobolev inequality and the logarithmic Sobolev inequality together with a series of delicate energy-derived uniform bounds. Furthermore, the existence of a least-energy solution of the limiting problem~\eqref{log:prob:intro} is also important.  In particular, we show the following. 

\begin{theorem}\label{main:thm:2}Let $N\geq 1$ and let $\Omega\subset \R^N$ be a bounded open Lipschitz set. For every $\mu\in(0,\frac{4}{N})$, the problem
 \begin{align}\label{lambda:problem}
  L_\Delta u = \mu\ln(|u|)u\quad \text{ in }\Omega,\qquad u_0=0\quad \text{ in }\R^N\backslash \Omega,
 \end{align}
 has a least-energy solution $u\in \mH(\Omega)\backslash\{0\}$ and
 \begin{align}\label{mountain:pass}
 J_0(u)=\inf_{\cN_0} J_0 = \inf_{\sigma\in \mathcal T}\max_{t\in[0,1]}J_0(\sigma(t))>0,
 \end{align}
 where $\mathcal T:=\{\sigma\in C^0([0,1],\mathbb H(\Omega)): \sigma(0)=0, \sigma(1)\neq 0, J_0(\sigma(1)\leq 0)\}$.
 Furthermore, all least-energy solutions of \eqref{lambda:problem} do not change sign in $\Omega$.
\end{theorem}
Here $J_0$ is defined in~\eqref{J0def} and $\cN_0$ is the \emph{Nehari manifold} associated to~\eqref{lambda:problem} (see~\eqref{N0:def} below). Equation~\eqref{mountain:pass} shows that the problem~\eqref{lambda:problem} also has a mountain-pass structure.  The Nehari manifold method is our main variational tool to show existence of solutions. The implementation of this approach, however, faces several difficulties in this setting. See Remark~\ref{diff:rmk} below for a discussion of some important differences between~\eqref{subcritical:intro} and~\eqref{lambda:problem}.  The most relevant obstacle to show Theorem~\ref{main:thm:2} is that the convergence of energy-minimizing sequences cannot be guaranteed by compact Sobolev embeddings. In particular, for the Hilbert space $\mH(\Omega)$ (defined in~\eqref{Hdef}), it is only known that
\begin{align}\label{comp:em}
\mathbb H(\Omega) \hookrightarrow L^2(\Omega)    \quad \text{ is compact}
\end{align}
(see \cite[Theorem 2.1]{CdP18} or \cite[Corollary 2.3]{LW20}), which alone is not enough to articulate an existence proof.

To compensate the loss of compactness we use a series of logarithmic inequalities.  Most prominently, the logarithmic Sobolev inequality
\begin{align}\label{si:intro}
  \frac{2}{N}\int_\Omega \ln(u^{2})|u|^2
  \leq\cE_L(u,u)+\frac{2}{N}\left(\ln\left(\int_\Omega|u|^2\right)+a_N\right)\int_\Omega|u|^2,\qquad u\in \mH(\Omega),
\end{align}
plays a crucial role.  The above inequality is a reformulation of \cite[Theorem~3]{B95} and it appears as a first order expansion of the standard Sobolev inequality at $s=0$ (see the proof of Proposition~\ref{log:prop} below).

As mentioned before, the notion of ``subcriticality'' of~\eqref{lambda:problem} is encoded in the assumption that $\mu<\frac{4}{N}$. Indeed, this fact, together with~\eqref{comp:em} and~\eqref{si:intro}, allows to obtain some compactness of energy minimizing sequences, see Proposition~\ref{prop:bd}.

\medskip

Theorem~\ref{main:thm} can be seen as a nonlinear analog of the results in \cite{FJW20}, where the asymptotic profile of $L^2$-normalized Dirichlet eigenfunctions is studied as $s\to 0^+$.  In particular, in \cite{FJW20} it is shown that, if $(s_n)\subset(0,\frac{1}{4})$ is such that $\lim_{n\to 0}s_n=0$, then, up to a subsequence,
\begin{align}\label{cLp}
\text{$\varphi_{k,s_n}\to \varphi_{k,L}$ in $L^p(\Omega)$ as $n\to\infty$ for any $p\in[1,\infty)$,}    
\end{align}
where $\varphi_{k,s_n}$ and $\varphi_{k,L}$ are the $k$-th eigenfunctions of $(-\Delta)^{s_n}$ and $L_\Delta$, respectively. Although the methods in \cite{FJW20} yield stronger results (compare, for instance, \eqref{cL2} and~\eqref{cLp}), they rely heavily on the linearity of the eigenvalue problem and therefore key arguments in \cite{FJW20} do not have a direct extension to the nonlinear setting. 
 
As far as we know, Theorem~\ref{main:thm:2} is the first result regarding existence of solutions to boundary value nonlinear problems involving the logarithmic Laplacian $L_\Delta$. Let us mention some previously known results regarding logarithmic operators.  The Dirichlet linear problem for $L_\Delta$ has been studied in \cite{CW19}, which also contains a Faber-Krahn type inequality and maximum principles in weak and strong forms. Further studies on the spectral properties (including sharp upper bounds for the Riesz means and lower bounds for the first Dirichlet eigenvalue of $L_\Delta$) can be found in \cite{LW20}; see also \cite{CV20} for upper and lower bounds for the sum of the first $k$ eigenvalues of $L_\Delta$. The Dirichlet problem for a logarithmic Schr\"odinger operator is considered in \cite{Feu21} and a \emph{regional logarithmic Laplacian} (associated to the regional fractional Laplacian in the small order limit) is studied in \cite{TW21}.

The logarithmic Laplacian also arises in the geometric context of the $0$-fractional perimeter, which has been studied recently in \cite{CdLNP21}.  In \cite{FKT20} an equation similar to~\eqref{lambda:problem} is considered with $\mu=\frac{4}{N}$ on the sphere and $L_\Delta$ is substituted with a conformally invariant logarithmic operator; in particular, the authors in \cite{FKT20} classify all nonnegative solutions of this equation, which are extremals of a Sobolev logarithmic inequality on the sphere presented in \cite[Theorem~3]{B95}.  From the probabilistic point of view, we mention that geometric stable stochastic processes with a logarithmic operator as an infinitesimal generator are studied in \cite{KM13}. Logarithmic-type energies also appear in the recent paper \cite{DCKNP21}, where the asymptotic analysis for the $s$-fractional heat flow as $s\to 0^+$ (and as $s\to 1^-$) is examined. 

Let us also comment on similar results to Theorem~\ref{main:thm} whenever $s_k\not\to 0$. If $s_k\to 1^-$, then a study of the nonlocal-to-local transition can be found in \cite{FBS20} (in the more general context of the fractional $p$-Laplacian).  See also \cite{BS19} and \cite{BS20} for convergence results in bounded and unbounded domains for subcritical Schrödinger equations. In the critical case, the only available result, as far as we know, is \cite{HSS21}, where convergence of solutions (in bounded and unbounded domains) is studied whenever $s_k\to s_0$ for any $s_0>0$ (including the higher order case $s_0>1$). In this setting, convergence of solutions (up to a subsequence) can be guaranteed in the $\cH^{s_0-\delta}_0$-sense for any $\delta\in(0,s_0)$.  Finally, if $v_s\in \cH^s_0(\Omega)$ is a solution of the fractional Poisson problem $(-\Delta)^sv_s=f$ in some bounded open set $\Omega$, then much more is known about the mapping $s\mapsto v_s$. In particular, in \cite{JSW20}, the logarithmic Laplacian $L_\Delta$ is used to characterize the continuity, differentiability, and monotonicity properties of the solution mapping $s\mapsto v_s$ for $s\in[0,1)$, where the case $s=0$ was previously studied in \cite{CW19}.

\medskip

The paper is organized as follows. In Section \ref{not:sec} we present the relevant definitions for the study of least-energy solutions. In Section \ref{aux:sec} we derive some Taylor expansions for the quadratic forms, give a full proof of the logarithmic Sobolev inequality \eqref{si:intro}, and show that $J_0$ is a $C^1$ functional. The uniform bounds for least-energy solutions are obtained in Section \ref{U:sec}. Section \ref{Sec:MT2} is devoted to the proof of Theorem \ref{main:thm:2} and Section \ref{6:sec} contains the proof of our main result Theorem \ref{main:thm}. Finally, in Section \ref{c:rmks} we include some closing remarks.

\section{Definitions and notations}\label{not:sec}

Let $\Omega\subset \R^N$ be an open bounded Lipschitz set. 
For $p\in[1,\infty]$ we use $L^p(\Omega)$ to denote the standard Lebesgue spaces with the norms
\begin{align*}
 |u|_p:=\left(\int_{\Omega}|u|^p\ dx\right)^{\frac{1}{p}}\quad\text{ for }p<\infty\qquad \text{ and }\qquad |u|_\infty:=\sup_{\Omega}|u|.
\end{align*}

The natural Hilbert space associated to~\eqref{subcritical:intro} is
\begin{align}\label{Hs:def}
\cH^s_0(\Omega):=\{u\in H^s(\R^N)\::\: u=0\text{ in }\R^N\backslash \Omega\},
\end{align}
where $H^s(\R^N)$ is the usual fractional Sobolev space.  We say that $u_s\in \cH^s_0(\Omega)$ is a (weak) solution of~\eqref{subcritical:intro} if 
\begin{align}\label{ws:intro}
 \cE_s(u_s,\varphi)=\int_\Omega |u_s|^{p_s-2}u_s\varphi\ dx\qquad\text{ for all } \varphi\in \cH^s_0(\Omega),
\end{align}
where
\begin{align*}
\cE_s(u,v)&:=\left(\frac{c_{N,s}}{2} \int_{\R^N}\int_{\R^N}\frac{(u(x)-u(y))(v(x)-v(y))}{|x-y|^{N+2s}}\ dxdy \right)^\frac{1}{2}
\end{align*}
is a scalar product in the Hilbert space $\cH^s_0(\Omega)$ with norm $\|u\|_s:=\cE(u,u)^\frac{1}{2}$. The \emph{energy functional} associated to~\eqref{subcritical:intro} is given by
\begin{align}\label{Js:def}
J_s:\cH^s_0(\Omega)\to\R,\qquad 
    J_s(u):=\frac{1}{2}\|u\|_s^2-I_s(u),\qquad I_s(u):=\frac{|u|_{p_s}^{p_s}}{p_s}=\frac{1}{p_s}\int_\Omega|u|^{p_s}\, dx.
\end{align}
Note that all nontrivial solutions of~\eqref{subcritical:intro} belong to the set
\begin{align}\label{cNs:def}
\cN_s:=\left\{u\in \cH^s_0(\Omega)\backslash\{0\}\::\: \|u\|_s^2 = |u|_{p_s}^{p_s}\right\}.
\end{align}
Then, a solution $u\in \cN_s$ is a \emph{least-energy solution} of~\eqref{subcritical:intro} if
\begin{align}\label{les}
J_s(u)=\inf_{v\in \cN_s}J_s(v),\quad \text{or, equivalently,}\quad 
 \|u\|_s^2
 =\inf_{v\in \cN_s} \|v\|^2_s,
\end{align}
see Theorem~\ref{existence}.

\medskip

On the other hand, the natural Hilbert space for the problem~\eqref{lambda:problem} is
\begin{align}\label{Hdef}
 \mH(\Omega):=\left\{
 u\in L^2(\R^N)\::\: \iint_{\substack{x,y\in\R^N\\|x-y|\leq 1}}
 \frac{|u(x)-u(y)|^2}{|x-y|^N}\ dx\,dy<\infty\text{ and }u=0\text{ in }\R^N\backslash \Omega
 \right\}
 \end{align}
with the scalar product
\begin{align}\label{cN:def}
  \cE(u,v):=\frac{c_N}{2}\iint_{\substack{x,y\in\R^N\\|x-y|\leq 1}}
 \frac{(u(x)-u(y))(v(x)-v(y))}{|x-y|^N}\ dx\,dy,\qquad c_N:=\pi^{-\frac{N}{2}}\Gamma\left(\tfrac{N}{2}\right)>0,
\end{align}
and the norm $\|u\|:=\left(\cE(u,u)\right)^\frac{1}{2}.$  The space of smooth functions with compact support in $\Omega$, denoted $C^\infty_c(\Omega)$, is dense in $\mH(\Omega)$, see \cite[Theorem 3.1]{CW19}.

The operator $L_\Delta$ has the following associated quadratic form
\begin{align}\label{cE}
 \cE_L(u,v):=\cE(u,v)-c_N\iint_{\substack{x,y\in\R^N\\|x-y|\geq 1}}\frac{u(x)v(y)}{|x-y|^N}\ dxdy+\rho_N\int_{\R^N} uv\ dx,
\end{align}
where 
\begin{align}\label{rho:def}
\rho_N:=2\ln 2+\psi(\tfrac{N}{2})+\gamma,\qquad \gamma:=-\Gamma'(1).
\end{align}
Here the constant $\gamma$ is known as the Euler-Mascheroni constant and $\psi:=\frac{\Gamma'}{\Gamma}$ is the digamma function. By \cite{CW19} (see also \cite{FJW20}), it holds that
\begin{align}\label{fou}
\cE_L(u,u)=\int_{\R^N}2\ln|\xi||\widehat u(\xi)|^2\ d\xi\qquad \text{ for all }u\in C^\infty_c(\Omega),
\end{align}
where $\widehat u$ denotes the Fourier transform of $u$ given by
\begin{align*}
    \widehat u(\xi)=\frac{1}{(2\pi)^\frac{N}{2}}\int_{\R^N}e^{-ix\cdot \xi}u(x)\ dx,\qquad \xi\in\R^N.
\end{align*}
Let $\mu\in(0,\frac{N}{4})$. We say that $u\in\mathbb H(\Omega)$ is a (weak) solution of~\eqref{lambda:problem} if 
\begin{align*}
    \cE_L(u,v)=\mu\int_\Omega uv\ln|u|\ dx\qquad \text{ for all }v\in \mH(\Omega).
\end{align*}
The \emph{energy functional} associated to~\eqref{lambda:problem} is given by
\begin{align}\label{J0def}
J_0:\mH(\Omega)\to\R,\qquad 
    J_0(u):=\frac{1}{2}\cE_L(u,u)-I(u),\qquad I(u):=\frac{\mu}{4}\int_\Omega u^2(\ln(u^2)-1)\ dx.
\end{align}
The functional $J_0$ is well defined in virtue of the logarithmic Sobolev inequality (see Proposition~\ref{log:prop}). Moreover, we show in Subsection \ref{difsubsec} that $J_0$ is of class $C^1$ in $\mH(\Omega)$.

All nontrivial solutions of~\eqref{subcritical:intro} belong to the set
\begin{align}\label{N0:def}
\cN_0:=\left\{u\in \mH(\Omega)\backslash\{0\}\::\: \cE_L(u,u) = \mu\int_\Omega u^2\ln|u|\ dx\right\}.
\end{align}

A solution $u\in \cN_0$ is a \emph{least-energy solution} of~\eqref{lambda:problem} if
\begin{align}\label{e2}
J_0(u) =\inf_{v\in \cN_0}J_0(v)\quad \text{ or, equivalently,}\quad |u|_2^2=\inf_{v\in \cN_0}|v|_2^2.
\end{align}

\begin{remark}\label{diff:rmk} Let us point out some important differences between the study of \eqref{subcritical:intro}  and \eqref{lambda:problem}.
\begin{enumerate}
    \item The continuity and differentiability of $J_s$ is a direct consequence of Sobolev embeddings and Hölder's inequality. This is not the case for $J_0$. Note that Hölder's inequality cannot be used to bound $I$ (the nonlinear part of $J_0$). Therefore, it is not immediate that $J_0$ is of class $C^1$ (or even defined) in $\mH(\Omega)$. We show in the next section that $J_0$ is in fact $C^1$ using a different argument. 
    \item Note that $I_s$ defined in~\eqref{Js:def} is weakly continuous by the compactness of the Sobolev embedding $\cH^s_0(\Omega)\hookrightarrow L^{p_s}(\Omega)$. This plays an important role to show the existence of solutions (see Theorem~\ref{existence}). But the embedding~\eqref{comp:em} is not enough to guarantee that $I$ is weakly continuous and this is one of the main technical obstacles to show Theorem~\ref{main:thm:2}.
    \item In~\eqref{N0:def} note that $\cE_L(u,u)$ and $\int_\Omega u^2\ln|u|\ dx$ could be negative or zero even if $u\neq 0$, whereas in~\eqref{cNs:def} the fact that $\|u\|_s$ and $|u|_{p_s}$ are positive for $u\neq 0$ provides the natural projection 
    \begin{align*}
        \left(\frac{\|u\|_s^2}{|u|_{p_s}^{p_s}}\right)^\frac{1}{p_s-2}u\in\cN_s\qquad \text{ for all }u\in\cH_0^s(\Omega)\backslash\{0\}.
    \end{align*}
    We show in Lemma \ref{lem:A1A2} that the projection to $\cN_0$ has to be given in terms of an exponential function.
        \item In~\eqref{les} the least-energy solution is characterized in terms of the norm in $\cH^s_0(\Omega)$, whereas in~\eqref{e2} we only have information about the norm in $L^2(\Omega)$. This implies that an energy bound yields a weaker control on the minimizing sequences for $J_0$; therefore, additional new arguments are needed to improve these estimates, see Proposition \ref{prop:bd}. 
\end{enumerate}
\end{remark}

\section{Auxiliary results}\label{aux:sec}

\subsection{Asymptotics for the best Sobolev constant}

\begin{lemma}\label{lem:lim}
If $a\neq 0$, then
\begin{align}\label{lim0}
 \lim_{s\to 0^+}\left(
 1+sa+o(s) \right)^{\frac{1}{s}}=e^{a}=\lim_{s\to 0^+}\left(
 1+sa \right)^{\frac{1}{s}}.
\end{align}
\end{lemma}
\begin{proof}
We claim that
\begin{align}\label{lim}
 \lim_{s\to 0^+}\left(
 1+s(a+o(1)) \right)^{\frac{1}{s}}-\left(
 1+sa \right)^{\frac{1}{s}}=0.
\end{align}
For $s\in(0,1)$, let $f:\R\to\R$ be given by $f(x)=(1+sx)^\frac{1}{s}$, then $f'(x)=(1+sx)^{\frac{1}{s}-1}$ and
\begin{align*}
 f(x+h)-f(x)
&= h\int_0^1 (1+s (x+\tau h))^{\frac{1}{s}-1}\ d\tau
= h\int_0^1 (s\tau h + sx+1)^{\frac{1}{s}-1}\ d\tau;
\end{align*}
but, since $|h|<1$, $\tau<1$, and $s\in(0,1)$,
\begin{align*}
 \left|\int_0^1 (s\tau h + sx+1)^{\frac{1}{s}-1}\ d\tau\right|
 &\leq  \frac{|s(1 +|x|)+1|^{\frac{1}{s}}}{|s(1 +|x|)+1|}=e^{1+|x|}+o(1)<C\quad \text{as $s\to 0^+$,}
\end{align*}
where $C>0$ is independent of $s\in(0,1)$.  Therefore, for any $x\in \R$, $s\in(0,1)$, and $|h|<1$, $|f(x+h)-f(x)| \leq Ch.$ This implies~\eqref{lim} and therefore~\eqref{lim0} holds.
\end{proof}

The next theorem is the fractional Sobolev inequality, see for example \cite[Theorem~1.1]{CT04}.

 \begin{theorem}[Fractional Sobolev inequality]\label{thm:sobolev} 
 Let $N\geq 1$, $s\in(0,\frac{N}{2})$, and $2^*_s:=\frac{2N}{N-2s}$. Then
 \begin{align*}
 |u|^2_{2^*_s}\leq \kappa_{N,s}\|u\|^2_{s}\qquad \text{ for all $u\in H^s(\R^N),$}
 \end{align*}
 where
\begin{equation}\label{eq:best_constant}
\kappa_{N,s}=2^{-2s}\pi^{-s}\frac{\Gamma(\frac{N-2s}{2})}{\Gamma(\frac{N+2s}{2})} \left(\frac{\Gamma(N)}{\Gamma(\frac{N}{2})}\right)^{\frac{2s}{N}}.
\end{equation}
\end{theorem}

Observe that the best Sobolev constant $\kappa_{N,s}$ is well behaved as $s\to 0^+$, in fact, $\lim_{s\to 0^+}\kappa_{N,s}=1$ and
\begin{align}
\lim_{s\to 0^+}\kappa_{N,s}^{\frac{1}{s}}
&=\frac{1}{4\pi}\left(\frac{\Gamma(N)}{\Gamma(\frac{N}{2})}\right)^{\frac{2}{N}}
\lim_{s\to 0^+}\left(1+s\partial_s\frac{\Gamma(\frac{N-2s}{2})}{\Gamma(\frac{N+2s}{2})}\Big|_{s=0}+o(s) \right)^\frac{1}{s}
&=\frac{1}{4\pi}\left(\frac{\Gamma(N)}{\Gamma(\frac{N}{2})}\right)^{\frac{2}{N}}e^{-2\psi(\frac{N}{2})},\label{kappa:eq}
\end{align}
where we used Lemma~\ref{lem:lim}, Taylor's expansion, and the fact that
\begin{align*}
\partial_s\frac{\Gamma(\frac{N-2s}{2})}{\Gamma(\frac{N+2s}{2})}\Bigg|_{s=0}
=-\frac{\Gamma(\frac{N-2s}{2})( \psi(\frac{N-2s}{2})+\psi(\frac{N+2s}{2}) )}{\Gamma(\frac{N+2s}{2})}\Bigg|_{s=0}=-2\psi(\tfrac{N}{2}),
\end{align*}
where $\psi=\frac{\Gamma'}{\Gamma}$ is the digamma function.

\subsection{Bounds and expansions for \texorpdfstring{$\cE_s$}{Es} and \texorpdfstring{$\cE_L$}{EL}}

\begin{lemma}\label{ln:bd}
Let $\alpha$ and $\beta$ be such that $\beta>\alpha$, then
\begin{align*}
\ln(r^2)r^\alpha\leq \frac{2}{\beta-\alpha} r^\beta\qquad \text{ for all }r>1
\end{align*}
\end{lemma}
\begin{proof}
We claim that $f(r)=\frac{2}{\beta-\alpha}r^{\beta-\alpha}-\ln(r^2)>0$ in $[1,\infty)$. Indeed, this follows since 
$f'(r) = 2r^{\beta-\alpha-1}-\frac{2}{r}=0$ if and only if 
$r=1$, which is the unique minimum of $f$ and $f(1)\geq 0$.
\end{proof}

\begin{lemma}\label{lem:5.5}
Let $\Omega\subset \R^N$ be an open bounded set and let $u,v\in L^2(\Omega)$ be such that $v=u=0$ in $\R^N\backslash \Omega$, then
\begin{align*}
    \iint_{\substack{x,y\in\R^N\\|x-y|\geq 1}}\frac{|u(x)v(y)|}{|x-y|^N}\ dxdy
    &\leq |u|_1|v|_1
    \leq |\Omega| |u|_2|v|_2,\\
    \iint_{\substack{x,y\in\R^N\\|x-y|\geq 1}}\frac{|u(x)u(y)|}{|x-y|^N}\ dxdy
    &\leq |u|_1^2
    \leq |\Omega| |u|_2^2.
\end{align*}
\end{lemma}
\begin{proof}
Note that
\begin{align*}
\int_{\R^N}\int_{\R^N\backslash B_1(y)}\frac{|u(x)v(y)|}{|x-y|^N}\ dx\,dy
&=\int_{\R^N}|v(y)|\int_{\R^N\backslash B_1(0)}\frac{|u(x+y)|}{|x|^N}\ dx\,dy\\
&\leq\int_{\R^N}|v(y)|\int_{\R^N}|u(x+y)|\ dx\,dy=|u|_1|v|_1.
\end{align*}
The result now follows from Hölder's inequality.
\end{proof}

\begin{lemma}\label{lem:cEbd}
Let $u\in\cH_0^s(\Omega)$ for some $s\in(0,1)$. Then $u\in \mH(\Omega)$ and there are $C_1=C_1(N)>0$ and $C_2=C_2(\Omega)>0$ such that
\begin{align*}
|\cE_L(u,u)|<C_1 |u|_1^2+\frac{1}{s}\|u\|^2_s\qquad\text{ and }\qquad \|u\|^2<C_2 |u|_2^2+\frac{1}{s}\|u\|^2_s.
\end{align*}
\end{lemma}
\begin{proof}
Let $u\in C^\infty_c(\Omega)$ and $C=\int_{B_1(0)} |\ln(|\xi|^2)|\ d\xi.$ Then, by~\eqref{fou} and Lemma~\ref{ln:bd} (with $\alpha=0$ and $\beta=2s$),
\begin{align*}
|\cE_L(u,u)|
&\leq \int_{\R^N} |\ln(|\xi|^2)||\widehat u(\xi)|^2\ d\xi\\
&\leq |\widehat u|_\infty^2\int_{B_1(0)} |\ln(|\xi|^2)|\ d\xi
+\int_{\R^N\backslash B_1(0)} \ln(|\xi|^2)|\widehat u(\xi)|^2\ d\xi\\
&\leq C |u|_1^2+s^{-1}\cE_s(u,u)=C |u|_1^2+s^{-1}\|u\|^2_s.
\end{align*}
Note that, by Lemma~\ref{lem:5.5}, there is $C'=C'(\Omega)>0$ such that $|\cE_L(u,u)|\geq\|u\|^2-C'|u|_2^2$ and therefore
\begin{align*}
\|u\|^2\leq (C'+C|\Omega|)|u|_2^2+s^{-1}\|u\|^2_s.
\end{align*}
The result for $u\in \cH^s_0(\Omega)$ follows by a standard density argument.
\end{proof}

\begin{lemma}\label{lem:sig:u}
Let $s\in(0,\frac{1}{4})$ and let $u_s\in \cH^s_0(\Omega)$ satisfy that
\begin{align}\label{bd:hyp}
     \|u_s\|_s<C
\end{align}
for some constant $C>0$. Then $u_s\in\cH^\sigma_0(\Omega)$ for $\sigma\in(0,s)$ and there is $C'=C'(C,\Omega)>0$ such that 
\begin{align}\label{sigma:bd}
|u_s|_2^2+\|u_s\|^2_\sigma<C'\qquad \text{ for all $\sigma\in(0,s)$.}  
\end{align}
\end{lemma}
\begin{proof}
Note that
\begin{align*}
\|u_s\|^2_\sigma
&=\int_{\R^N}|\xi|^{2\sigma}|\widehat u_s(\xi)|^{2}\ d\xi
=\int_{\{|\xi|\leq 1\}}|\xi|^{2\sigma}|\widehat u_s(\xi)|^{2}\ d\xi
+\int_{\{|\xi|>1\}}|\xi|^{2\sigma}|\widehat u_s(\xi)|^{2}\ d\xi\\
&\leq\int_{\R^N}|\widehat u_s(\xi)|^{2}\ d\xi
+\int_{\R^N}|\xi|^{2s}|\widehat u_s(\xi)|^{2}\ d\xi=|u_s|_2^2+\|u_s\|_s^2<C^2+|u_s|_2^2.
\end{align*}
The claim now follows by Hölder's inequality, because
\begin{align*}
|u_s|_2^2\leq |\Omega|^\frac{p_s-2}{p_s} |u_s|_{p_s}^{2}
=|\Omega|^\frac{p_s-2}{p_s} \|u_s\|_{s}^{\frac{4}{p_s}}
\leq \max_{t\in[0,\frac{1}{4})}|\Omega|^\frac{p_t-2}{p_t} C^{\frac{4}{p_t}}.
\end{align*}
\end{proof}

\begin{lemma}\label{Es:exp:lem}
Let $0<\sigma<s<1$ and $u\in \cH_0^s(\Omega)$. There is a constant $d_N>0$ depending only on $N$ such that
\begin{align*}
    \left| \cE_\sigma(u,u)-|u|_2^2-\sigma\cE_L(u,u) \right|
    \leq d_N \frac{\sigma^2}{(s-\sigma)^2}\left( |u|_1^2+\cE_s(u,u) \right).
\end{align*}
In particular, $\|u\|_\sigma\to |u|_2$ as $\sigma\to 0^+$.
\end{lemma}
\begin{proof}
We argue as in \cite[Lemma 2.6]{FJW20}.   For $\xi\in\R^N$, let $h(\sigma):=|\xi|^{2\sigma}$. Then 
$h'(\sigma)=|\xi|^{2\sigma}\ln(|\xi|^2)$, 
$h''(\sigma)=|\xi|^{2\sigma}(\ln(|\xi|^2))^2$, and
\begin{align*}
    \left||\xi|^{2\sigma}-1-\sigma\ln(|\xi|^2)\right|
    &=|h(\sigma)-h(0)-\sigma h'(0)|
    =\left|\int_0^t h''(\tau)(\sigma-\tau)\ d\tau\right|\\
    &\leq |\ln(|\xi|^2)|^2\int_0^\sigma |\xi|^{2\tau}|\sigma-\tau|\ d\tau
    =|\ln(|\xi|^2)|^2 \sigma^2 \int_0^1 |\xi|^{2\tau \sigma}|1-\tau|\ d\tau\\
    &\leq  \sigma^2|\ln(|\xi|^2)|^2 (\chi_{\{|\xi|\leq 1\}}+|\xi|^{2\sigma}\chi_{\{|\xi|> 1\}}),
\end{align*}
since $|\xi|^{2\tau \sigma}|1-\tau|\leq 1$ if $|\xi|\leq 1$ and $|\xi|^{2\tau \sigma}|1-\tau|<|\xi|^{2\sigma}$ if $|\xi|>1$ for $\tau\in(0,1)$.

Note also that, for $|\xi|>1$ and $0<\sigma<s<1$, we have that $2\ln|\xi|<\frac{2}{s-\sigma}|\xi|^{s-\sigma}$, by Lemma~\ref{ln:bd}, and then 
\begin{align}\label{xi}
|\ln(|\xi|^2)|^2|\xi|^{2\sigma}<\frac{4}{(s-\sigma)^2}|\xi|^{2s}\quad \text{ for }|\xi|>1.
\end{align}

Let $u\in C^\infty_c(\Omega)$ and $d_N:=4+(2\pi)^{-N}\int_{B_1(0)}|\ln(|\xi|^2)|^2\ d\xi$.  Then, by~\eqref{xi},
\begin{align*}
\Big| \cE_\sigma(u,u)-|u|_2^2&-\sigma\cE_L(u,u) \Big|
=\left|\int_{\R^N}(|\xi|^{2\sigma}-1-\sigma\ln(|\xi|^2)|\widehat u(\xi)|^2\ d\xi \right|\\
&\leq \sigma^2 \int_{\R^N}
|\ln(|\xi|^2)|^2 (\chi_{\{|\xi|\leq 1\}}+|\xi|^{2\sigma}\chi_{\{|\xi|> 1\}})|\widehat u(\xi)|^2\ d\xi\\
&\leq \sigma^2 \|\widehat u\|^2_\infty \int_{B_1(0)}|\ln(|\xi|^2)|^2\ d\xi 
+ \sigma^2 \int_{\R^N\backslash B_1(0)}|\ln(|\xi|^2)|^2|\xi|^{2\sigma}|\widehat u(\xi)|^2\ d\xi\\
&\leq \sigma^2 (2\pi)^{-N}|u|_1^2\int_{B_1(0)}|\ln(|\xi|^2)|^2\ d\xi 
+ \sigma^2 \frac{4}{(s-\sigma)^2} \int_{\R^N\backslash B_1(0)}|\xi|^{2s}|\widehat u(\xi)|^2\ d\xi\\
&\leq d_N \frac{\sigma^2}{(s-\sigma)^2}\left(|u|^2_1 + \cE_s(u,u)\right).
\end{align*}
The result for $u\in \cH^s_0(\Omega)$ follows by density.
\end{proof}

\subsection{Sharp logarithmic Sobolev inequality}

Recall that
\begin{align*}
 |u|_2:=\left(\int_{\R^N}|u|^2\right)^\frac{1}{2}\qquad \text{ for }u\in L^2(\Omega).
\end{align*}
Now we present a sharp logarithmic Sobolev inequality, which is one of our main tools to guarantee the compactness of the minimizing sequences of $J_0$ (see \eqref{J0def}).  This result is shown in \cite[Theorem 3]{B95} for functions in the Schwarz space and using a different definition of the Fourier Transform (which influences the expression of the optimal constants that are important in our arguments).  For completeness, we include here a slightly different proof and a statement which is more adequate for our purposes. 
\begin{proposition}\label{log:prop}
For every $u\in \mH(\Omega)$,
\begin{align}\label{eq:ineq_becker}
  \frac{2}{N}\int_\Omega \ln(u^{2})u^2\ dx
  \leq\cE_L(u,u)+\frac{2}{N}\ln(|u|_2^2)|u|_2^2+a_N|u|^2_2,
\end{align}
where 
 \begin{align}\label{A}
  a_N:=\frac{2}{N} \ln \left(\frac{\Gamma (N)}{\Gamma
   \left(\frac{N}{2}\right)}\right)-\ln (4 \pi )-2 \psi\left(\frac{N}{2}\right)
 \end{align}
 and $\psi$ is the digamma function.
\end{proposition}
\begin{proof}
Let $u\in C^\infty_c(\R^N)\backslash\{0\}$ and note that $|u|^{2^*_s}_{2^*_s}
=|u|_2^2+s\frac{4}{N}\int_{\R^N} |u|^2\ln|u|\, dx+o(s)$ as $s\to 0^+$. Moreover, for a positive function $F\in C^1([0,1])$,
\begin{align*}
    \partial_sF(s)^\frac{2}{2^*_s}=F(s)^{\frac{N-2 s}{N}} \left(\frac{(N-2 s) F'(s)}{N F(s)}-\frac{2 \ln (F(s))}{N}\right),
\end{align*}
and therefore, using $F(s)=|u|_{2^*_s}^{2^*_s}$, we obtain that $F(0)=|u|_2^2$ and
\begin{align}\label{exp1}
|u|_{2^*_s}^2
&=|u|_2^2+s\partial_t F(t)^\frac{2}{2^*_t}\Big|_{t=0}+o(s)
=|u|_2^2+s\frac{4}{N}\left(
\int_{\R^N} |u|^2\ln|u|\ dx-|u|_2^2 \ln |u|_2
\right)
+o(s)
\end{align}
 as $s\to 0^+$.
 Let $\kappa_{N,s}$ be given by~\eqref{eq:best_constant} and $a_N$ by~\eqref{A}, then $\kappa_{N,s}=1+s a_N+o(s)$ as $s\to 0^+$. Furthermore, by Lemma~\ref{Es:exp:lem}, $\|u\|_s^2=|u|^2_2+s\cE_L(u,u)+o(s)$ as $s\to 0^+$ and therefore
\begin{align}\label{exp2}
\kappa_{N,s}\|u\|^{2}_s
=|u|_2^2+s(a_N|u|_2^2+\cE_L(u,u))
+o(s)\quad \text{ as }s\to 0^+.
\end{align}
 By Theorem~\ref{thm:sobolev}, $|u|^2_{2^*_s}\leq\kappa_{N,s}\|u\|^{2}_s$, and, using~\eqref{exp1} and~\eqref{exp2},
\begin{align*}
\frac{4}{N}\left(
\int_{\R^N} |u|^2\ln|u|\ dx-|u|_2^2 \ln |u|_2
\right)
\leq 
    a_N|u|_2^2+\cE_L(u,u).
\end{align*}
This yields the claim for $u\in C^\infty_c(\R^N)$ and the general statement for $u\in \mH(\Omega)$ follows by density. Indeed, let $u\in\mH(\Omega)$ and let $(u_n)_{n\in\N}\subset C^\infty_c(\Omega)$ be such that $u_n\to u$ in $\mH(\Omega)$ ($C^\infty_c(\Omega)$ is dense in $\mH(\Omega)$ by \cite[Theorem 3.1]{CW19}). By~\eqref{comp:em}, $u_n\to u$ in $L^2(\Omega)$, and therefore
\begin{align}\label{Fa0}
\frac{4}{N}\left(
\int_{\R^N} |u_n|^2\ln|u_n| \ dx-|u|_2^2 \ln |u|_2
\right)
\leq 
    a_N|u|_2^2+\cE_L(u,u)+o(1)\quad \text{as $n\to\infty$}.
\end{align}
It suffices to show that, up to a subsequence,
\begin{align}\label{Fa}
\int_{\R^N} |u|^2\ln|u| \ dx\leq \lim_{n\to \infty}\int_{\R^N}|u_n|^2  \ln|u_n|\ dx\qquad \text{ for all }n\in\N.
\end{align}
Using that $\Omega$ is bounded and the dominated convergence theorem, we have that, passing to a subsequence,
\begin{align}\label{Fa1}
    \lim_{n\to\infty}\int_{\{|u_n|\leq 1\}} |u_n|^2\ln|u_n|\ dx =\int_{\{|u|\leq 1\}} |u|^2\ln|u|\ dx,
\end{align}
whereas, by Fatou's Lemma, we deduce that, passing to a subsequence,
\begin{align}\label{Fa2}
\int_{\{|u|\geq 1\}} |u|^2\ln|u|\ dx\leq \lim_{n\to\infty}\int_{\{|u_n|\geq 1\}} |u_n|^2\ln|u_n|\ dx.
\end{align}
Then~\eqref{Fa1} and~\eqref{Fa2} imply~\eqref{Fa}, and the logarithmic Sobolev inequality for general $u\in\mH(\Omega)$ follows from~\eqref{Fa} and~\eqref{Fa0}.
\end{proof}

\subsection{Differentiability of the energy functional}\label{difsubsec}

Let $J_0$ and $I$ be given by~\eqref{J0def}.  We show that $J_0$ is of class $C^1$ in $\mH(\Omega)$. Recall that 
\begin{align}\label{cEuu}
\cE_L(u,u)=\|u\|^2-c_N\iint_{\substack{x,y\in\R^N\\|x-y|\geq 1}}\frac{u(x)u(y)}{|x-y|^N}\ dxdy+\rho_N|u|_2^2.
\end{align}
It is clear that $\|\cdot\|$ is differentiable and $|\cdot|_2$ is also differentiable by~\eqref{comp:em}. Moreover, 
\begin{align*}
B(u,v):=\iint_{\substack{x,y\in\R^N\\|x-y|\geq 1}}\frac{u(x)v(y)}{|x-y|^N}\ dxdy
\end{align*}
is a bounded bilinear form (by Lemma~\ref{lem:5.5}). We show next the differentiability of $I$. 

\begin{lemma}\label{I:C1}
Let $I$ be given by~\eqref{J0def}. Then $I$ is of class $C^1$ in $\mH(\Omega)$ and $I'(u)v=\mu\int_\Omega uv\ln|u|$. In particular, $I'(u)\in\cL(\mH(\Omega),\R)$ and $I':\mH(\Omega)\to\cL(\mH(\Omega),\R)$ is continuous.
\end{lemma}
\begin{proof}
Let $x\in \Omega$, $u,v\in\mH(\Omega)$, and $\delta\in(-1,1)\backslash\{ 0\}$. Let $h(t):=t^2(\ln(t^2)-1)$, $h'(t)=2t\ln|t|$. By the mean value theorem, there is $\tau=\tau(x)\in[0,1]$ such that
\begin{align*}
Q_\delta(x):=\frac{h(u(x)+\delta v(x))-h(u(x))}{\delta}
&=2(u(x)+\delta\tau v(x))v(x)\ln|u(x)+\tau \delta v(x)|.
\end{align*}

Assume first that $|u(x)+\tau \delta v(x)|\geq 1$, then 
\begin{align*}
1\leq |u(x)+\delta\tau v(x)|\leq |u(x)|+|v(x)|\leq 2\max\{|u(x)|,|v(x)|\}    
\end{align*}
and, since $h'$ is monotone increasing in $(1,\infty)$, 
\begin{align*}
|Q_\delta(x)|
&\leq 4\max\{|u(x)|,|v(x)|\}|v(x)||\ln(2\max\{|u(x)|,|v(x)|\})|\\
&\leq 4\max\{|u(x)|,|v(x)|\}^2|\ln(2\max\{|u(x)|,|v(x)|\})|\\
&\leq |2u(x)|^2|\ln|2u(x)||
+|2v(x)|^2|\ln|2v(x)||=:M(x).
\end{align*}
Note that $M\in L^1(\Omega)$.  Indeed, by~\eqref{comp:em}, Lemma~\ref{lem:5.5}, and Proposition~\ref{log:prop}, 
\begin{align*}
\int_\Omega w^2|\ln |w||
&=\int_{\{|w|\geq 1\}} w^2\ln |w|-\int_{\{|w|<1\}} w^2\ln |w|=\int_{\Omega} w^2\ln |w|-2\int_{\{|w|<1\}} w^2\ln |w|\\
&\leq \frac{N}{2}\cE_L(w,w)+\ln(|w|_2^2)|w|_2^2+\frac{N}{2}|a_N||w|_2^2+2|\Omega|\sup_{t\in(0,1)}t^2|\ln t|
<\infty
\end{align*}
for all $w\in\mH(\Omega)$.

On the other hand, if $|u(x)+\tau \delta v(x)|<1$, then $|Q_\delta(x)|\leq 2e^{-1}$, because $\sup_{(0,1)}|h'|\leq 2e^{-1}$. Then $|Q_\delta|\leq M+2e^{-1}$ in $\Omega$ and, by the dominated convergence theorem,
\begin{align*}
I'(u)v=\lim_{\delta\to 0}\frac{I(u+\delta v)-I(u)}{\delta}
=\mu\int_\Omega uv\ln|u|\ dx.
\end{align*}
Using similar arguments, one can show that, for $u,v,w\in\mH(\Omega)$,
\begin{align*}
\lim_{\delta\to 0}|I'(u)(v+\delta w)-I'(u)v|
&\leq \frac{\mu}{2}\lim_{\delta\to 0}\delta \int_\Omega 
|u w\ln(u^2)|=0,\\
\lim_{\delta\to 0}|I'(u+\delta w)v-I'(u)v|
&\leq \frac{\mu}{2} \lim_{\delta\to 0}\int_\Omega 
|(u+\delta w)\ln((u+\delta w)^2)-u\ln(u^2)||v|=0,
\end{align*}
and therefore $I'(u)\in\cL(\mH(\Omega),\R)$ and $I':\mH(\Omega)\to\cL(\mH(\Omega),\R)$ is continuous.
\end{proof}

\section{Uniform bounds for the elements in the Nehari manifold}\label{U:sec}

In this section we show some uniform estimates for every $u\in\cN_s$ with $s\in(0,\frac{1}{4})$.  In particular, we show in Proposition~\ref{cor:Ebd} that all least-energy solutions of~\eqref{subcritical:intro} are uniformly bounded in $\mH(\Omega)$.

\medskip

We begin with some auxiliary lemmas. 

\begin{lemma}\label{lem:Ebd}
Let $p\in C^1([0,\tfrac{1}{4}])$ satisfy~\eqref{p:hyp}. There is a constant $c=c(p,\Omega)>0$ such that 
\begin{align}\label{bds}
\|u\|_s>c\quad \text{for all $u\in\cN_s$ and all $s\in(0,\tfrac{1}{4}).$}
\end{align}
\end{lemma}
\begin{proof}
Let $F_s:\cH^s_0(\Omega)\backslash\{0\}\to \R$ be given by $F_s(u)=\|u\|_s^2 - |u|_{p_s}^{p_s}.$ By Hölder's inequality and Theorem~\ref{thm:sobolev},
\begin{align*}
F_s(u)\geq \|u\|_s^2 - |\Omega|^\frac{2^*_s-p_s}{2^*_s}|u|_{2^*_s}^{p_s}
\geq \|u\|_s^2 - |\Omega|^\frac{2^*_s-p_s}{2^*_s}\kappa_{N,s}^{\frac{p_s}{2}}\|u\|_{s}^{p_s}
=\|u\|_s^2(1 - |\Omega|^\frac{2^*_s-p_s}{2^*_s}\kappa_{N,s}^{\frac{p_s}{2}}\|u\|_{s}^{p_s-2}).
\end{align*}

Let $g(t,s) :=1 - |\Omega|^\frac{2^*_s-p_s}{2^*_s}\kappa_{N,s}^{\frac{p_s}{2}}t^{p_s-2},$ where $\kappa_{N,s}$ is given in~\eqref{eq:best_constant}. Then
\begin{align*}
g(t,s)>0\quad \text{ if }\quad t<|\Omega|^\frac{2^*_s-p_s}{2^*_s(2-p_s)} \kappa_{N,s}^{\frac{p_s}{{2(2-p_s)}}}.      
\end{align*}
Note that
\begin{align*}
\frac{2^*_s-p_s}{2^*_s(2-p_s)}
=\frac{2^*_s-2}{2^*_s(2-p_s)}+\frac{1}{2^*_s}
=-\frac{4}{(N-2s)2^*_s \int_0^1 p'(\tau s)\,d\tau}+\frac{1}{2^*_s}
\to \frac{1}{2}-\frac{2}{Np'(0)}\quad \text{as $s\to 0$,}
\end{align*}
and therefore
\begin{align*}
\lim_{s\to 0}|\Omega|^\frac{2^*_s-p_s}{2^*_s(2-p_s)}=
|\Omega|^{\frac{1}{2}-\frac{2}{Np'(0)}}>0.
\end{align*}
Furthermore, by~\eqref{kappa:eq},
\begin{align*}
\lim_{s\to 0}\kappa_{N,s}^{\frac{p_s}{{2(2-p_s)}}}=\lim_{s\to 0}
\left(
\kappa_{N,s}^\frac{1}{s}
\right)^{\frac{s p_s}{{2(2-p_s)}}}
= 
\left(\frac{1}{4\pi}\left(\frac{\Gamma(N)}{\Gamma(\frac{N}{2})}\right)^{\frac{2}{N}}e^{-2\psi(\frac{N}{2})}\right)^{-\frac{1}{p'(0)}}
>0.
\end{align*}
 As a consequence, there is $c=c(p,\Omega)>0$ such that $F_s(u)>0$ if $\|u\|_s\in(0,c)$, and then $\|u\|_s>c$ for all $u\in \cN_s$ and for all $s\in(0,\frac{1}{4})$, as claimed.  \end{proof}

\begin{lemma}\label{lem:t0}
Let $s\in(0,\tfrac{1}{4})$, $\varphi\in C^\infty_c(\Omega)\backslash\{0\}$, and $p_s:=p(s)$, where $p$ satisfies~\eqref{p:hyp}.  Let $t_\varphi^s$ be given by
\begin{align}\label{ts:def}
 t_\varphi^s :=\left( \frac{\|\varphi\|_s^2}{|\varphi|_{p_s}^{p_s}}\right)^{\frac{1}{p_s-2}}.
\end{align}
Then $t^s_\varphi \varphi\in\cN_s$ and
\begin{align*}
 \lim_{s\to 0^+} t_\varphi^s =
  t_\varphi^0
 :=\exp\left(\frac{\cE_L(\varphi,\varphi)-p'(0)\int_\Omega\ln (|\varphi|)|\varphi|^2\ dx}{p'(0)|\varphi|_2^2}\right)>0.
\end{align*}
In particular, $\sup_{s\in(0,\frac{1}{4})}t_\varphi^s<\infty$.
\end{lemma}
\begin{proof}
Let $s\in(0,\tfrac{1}{4})$ and  $\varphi\in C^\infty_c(\Omega)\backslash\{0\}$. Then, by Lemma~\ref{Es:exp:lem}, $\|\varphi\|_s^2=|\varphi|_2^2+s\cE_L(\varphi,\varphi)+o(s)
$ as $s\to 0^+.$ On the other hand, $|\varphi|_{p_s}^{p_s}=|\varphi|^2_2
+
s p'(0)\int_\Omega |\varphi|^2\ln |\varphi|\, dx+o(s).$ 
Let $A=|\varphi|_2^2$. Then, by Lemma~\ref{lem:lim},
\begin{align*}
 \lim_{s\searrow 0}t_\varphi^s 
 &= \lim_{s\searrow 0}\left( 
 \frac{1 + A^{-1}s\cE_L(\varphi,\varphi) + o(s)}
 {1+s\frac{p'(0)}{A} \int_\Omega |\varphi|^2\ln |\varphi|\, dx+o(s)}
 \right)^{\frac{1}{p_s-2}}\\
 &= \left(\frac{\lim_{s\searrow 0}(1 + A^{-1}s\cE_L(\varphi,\varphi) + o(s))^\frac{1}{s}}
 {\lim_{s\searrow 0}(
 {1+s\frac{p'(0)}{A} \int_\Omega |\varphi|^2\ln |\varphi|\, dx+o(s)})^\frac{1}{s}}
\right)^\frac{1}{p'(0)}
\\ 
 &= \left(\frac{
 e^{A^{-1}\cE_L(\varphi,\varphi)}
 }{
e^{p'(0)A^{-1} \int_\Omega |\varphi|^2\ln |\varphi|}}
\right)^\frac{1}{p'(0)} 
=\exp\left(\frac{\cE_L(\varphi,\varphi)-p'(0)\int_\Omega\ln (|\varphi|)|\varphi|^2}{p'(0)|\varphi|_2^2}\right)>0.
\end{align*}
This implies that the function $s\mapsto t^s_\varphi$ has a continuous extension to the compact set $[0,\tfrac{1}{4}]$, and therefore $\sup_{s\in(0,\frac{1}{4})}t_\varphi^s<\infty$, as claimed.
\end{proof}

\begin{remark}
In Lemma~\ref{lem:A1A2} below we show that, in fact, $t_\varphi^0\varphi\in \cN_0$ with $\mu=p'(0)$.
\end{remark}

The following result provides an ``intermediate" logarithmic-type Sobolev inequality.

\begin{lemma}\label{prebd}
Let $s\in(0,\tfrac{1}{4})$ and let $\varphi\in C^\infty_c(\Omega)$. It holds that
\begin{align*}
\int_0^1\frac{4N}{(N-2s\tau)^2}\int_\Omega|\varphi|^{2^*_{s\tau}}\ln|\varphi|\ dx\,d\tau\leq \int_0^1 k'(s\tau)\|\varphi\|_{s\tau}^{2^*_{s\tau}}
+2k(s\tau)\int_{\R^N}|\xi|^{2s\tau}\ln|\xi| |\widehat \varphi(\xi)|^2\ d\xi\,d\tau,
\end{align*}
where $k(s):=(\kappa_{N,s})^{\frac{2^*_s}{2}}$ and with $\kappa_{N,s}$ as in Theorem~\ref{thm:sobolev}. Moreover, if
\begin{align}\label{C:def}
\|\varphi\|_s^2<C\qquad \text{ for some $C>0$,}
\end{align}
then there is $C_1=C_1(C,\Omega)>0$ such that 
\begin{align*}
\int_0^1\int_{\{|\varphi|\geq 1\}}|\varphi|^{2^*_{s\tau}}\ln|\varphi|\ dx\,d\tau\leq C_1+\frac{N}{4}\int_0^1 k(s\tau)\int_{\R^N}|\xi|^{2s\tau}\ln(|\xi|^2) |\widehat \varphi(\xi)|^2\ d\xi\,d\tau.
\end{align*}
\end{lemma}
\begin{proof}
Let $s\in(0,\frac{1}{4})$, $k(s):=(\kappa_{N,s})^{\frac{2^*_s}{2}}$, and $G(s):=k(s)\|\varphi\|_s^{2^*_s}-|\varphi|_{2^*_s}^{2^*_s}.$ Then $k\in C^1([0,\tfrac{1}{4}])$ and $G\in C^1((0,\tfrac{1}{4}))$ with
\begin{align*}
G'(s)=k'(s)\|\varphi\|_s^{2^*_s}
+2k(s)\int_{\R^N}|\xi|^{2s}\ln|\xi| |\widehat \varphi(\xi)|^2\ d\xi
-\frac{4N}{(N-2s)^2}\int_\Omega|\varphi|^{2^*_s}\ln|\varphi|\, dx.
\end{align*}
Note that $\lim_{t\to 0^+}G(t)=0$ and, by Theorem~\ref{thm:sobolev}, $G(t)\geq 0$ for $t\in(0,\tfrac{1}{4})$. Then $\int_0^1 G'(s\tau)\ d\tau\geq 0$, namely,
\begin{align}
J&:=\int_0^1\frac{4N}{(N-2s\tau)^2}\int_\Omega|\varphi|^{2^*_{s\tau}}\ln|\varphi|\ dx\,d\tau\notag\\
&\leq \int_0^1 k'(s\tau)\|\varphi\|_{s\tau}^{2^*_{s\tau}}
+2k(s\tau)\int_{\R^N}|\xi|^{2s\tau}\ln|\xi| |\widehat \varphi(\xi)|^2\ d\xi\,d\tau.\label{it:0}
\end{align}
By Lemma~\ref{lem:sig:u} and \eqref{C:def}, there is $M=M(C,\Omega)>0$ such that $\|\varphi\|^2_{s\tau}<M$ for all $\tau\in(0,1)$; but then, there is $C'=C'(C,\Omega)>0$ such that 
\begin{align}\label{it:1}
\int_0^1|k'(s\tau)|\|\varphi\|_{s\tau}^{2^*_{s\tau}}\, d\tau<C'.
\end{align}
Moreover,
\begin{align}
J&=\int_0^1\frac{4N}{(N-2s\tau)^2}\int_{\Omega\cap \{|\varphi|\geq 1\}}|\varphi|^{2^*_{s\tau}}\ln|\varphi|\ dx\,d\tau+\int_0^1\frac{4N}{(N-2s\tau)^2}\int_{\Omega\cap\{|\varphi|<1\}}|\varphi|^{2^*_{s\tau}}\ln|\varphi|\ dx\,d\tau\notag\\
&\geq \frac{4}{N} \int_0^1\int_{\Omega\cap \{|\varphi|\geq 1\}}|\varphi|^{2^*_{s\tau}}\ln|\varphi|\ dx\,d\tau-|\Omega|c_1,\label{it:2}
\end{align}
where $c_1:=\sup_{t\in(0,1),\theta\in(0,\frac{1}{4})}\frac{4N}{(N-2\theta)^2} t^{2^*_{\theta}}|\ln|t||$.  By~\eqref{it:0},~\eqref{it:1}, and~\eqref{it:2},
\begin{align*}
\int_0^1\int_{\Omega\cap \{|\varphi|\geq 1\}}|\varphi|^{2^*_{s\tau}}\ln|\varphi|\ dx\,d\tau\leq 
\frac{N}{4}\left(
|\Omega|c_1+C'+\int_0^1 2k(s\tau)\int_{\R^N}|\xi|^{2s\tau}\ln|\xi| |\widehat \varphi(\xi)|^2\ d\xi\,d\tau
\right),
\end{align*}
and the claim follows.
\end{proof}

The next result shows that uniform bounds in $\cN_s$ yield uniform bounds in $\mH(\Omega)$.  Note that Lemma~\ref{lem:cEbd} provides such a bound only for $s$ far away from zero.  For $s$ near zero a finer analysis is required. 

\begin{lemma}\label{lem:vip}
Let $C_0>0$, $s\in(0,\frac{1}{4})$, $p_s=p(s)$ with $p$ as in~\eqref{p:hyp}, and let $\varphi\in \cN_s$ be such that $\|\varphi\|_s^2<C_0$. Then there is $C=C(C_0,p,\Omega)>0$ such that 
\begin{align*}
\|\varphi\|^2=\cE(\varphi,\varphi)<C.
\end{align*}
\end{lemma}
\begin{proof} Assume first that $\varphi\in C^\infty_c(\Omega)\cap \cN_s$. For $\tau\in(0,1)$ and $\sigma\in(0,\frac{1}{4})$, let 
\begin{align*}
h_\sigma(\tau):=1-\frac{N}{4}(\kappa_{N,\sigma\tau})^{\frac{2^*_{\sigma\tau}}{2}}\sup_{(0,\sigma)}|p'|,     
\end{align*}
where $\kappa_{N,s}$ is given in \eqref{kappa:eq}.
Note that $p'(0)\in(0,\frac{4}{N})$ (by \eqref{p:hyp}) and $\kappa_{N,\sigma}\to 1$ as $\sigma\to 0^+$ (by \eqref{kappa:eq}). Therefore there is $s_0\in(0,\frac{1}{4})$ such that, if $s\in(0,s_0)$, then
\begin{align}\label{s0}
p'(s\tau)>0\quad \text{ for }\tau\in(0,1)\qquad \text{ and }\qquad 
\delta:=\min_{\tau\in(0,1)}h_s(\tau)\in(0,1).
\end{align}
For $s\in[s_0,\frac{1}{4})$, the claim follows from Lemmas  \ref{lem:cEbd} and \ref{lem:sig:u}.

To show the claim in $(0,s_0)$, let $s\in(0,s_0)$.  Note that
\begin{align}
    \cI:=\frac{\|\varphi\|_s^2-|\varphi|_2^2}{s}
    =\int_0^1\int_{\R^N}|\xi|^{2s\tau}\ln(|\xi|^2)|\widehat \varphi(\xi)|^2\ d\xi \ d\tau.\label{Iphi}
\end{align}
On the other hand, using that $\varphi\in\cN_s$ and \eqref{s0},
\begin{align*}
    \cI&=\frac{|\varphi|_{p_s}^{p_s}-|\varphi|_2^2}{s}
    =\int_0^1 p'(s\tau) \int_\Omega |\varphi|^{p(s\tau)} \ln |\varphi|\ dx
    \ d\tau\\
    &=\int_0^1 p'(s\tau) \int_{\{|\varphi|<1\}} |\varphi|^{p(s\tau)} \ln |\varphi|\ dx
    \ d\tau
    +\int_0^1 p'(s\tau) \int_{\{|\varphi|\geq 1\}} |\varphi|^{p(s\tau)} \ln |\varphi|\ dx
    \ d\tau\\
    &\leq \sup_{(0,s)}|p'|\int_0^1 \int_{\{|\varphi|\geq 1\}} |\varphi|^{2^*_{s\tau}} \ln |\varphi|\ dx
    \ d\tau.
\end{align*}
Then, by Lemma~\ref{prebd}, there is $C_1=C_1(C_0,\Omega)>0$ such that
\begin{align}
\cI\leq C_1+\sup_{(0,s)}|p'|\frac{N}{4}\int_0^1 (\kappa_{N,s\tau})^{\frac{2^*_{s\tau}}{2}}\int_{\R^N}|\xi|^{2s\tau}\ln(|\xi|^2) |\widehat \varphi(\xi)|^2\ d\xi\,d\tau.\label{Iphi2}
\end{align}
Furthermore, by~\eqref{Iphi} and~\eqref{Iphi2},
\begin{align}\label{it:3}
\int_0^1\int_{\R^N}h_s(\tau)|\xi|^{2s\tau}\ln|\xi|^2 |\widehat \varphi(\xi)|^2\ d\xi\,d\tau\leq C_1.
\end{align}
Note that, if $C_2:=\int_{\{|\xi|\leq 1\}}|\ln|\xi|^2| \ d\xi$,
\begin{align}
\int_0^1\int_{\{|\xi|\leq 1\}}h_s(\tau)|\xi|^{2s\tau}\ln|\xi|^2 &|\widehat \varphi(\xi)|^2\ d\xi\,d\tau\notag\\
&=-\int_0^1\int_{\{|\xi|\leq 1\}}h_s(\tau)|\xi|^{2s\tau}|\ln|\xi|^2| |\widehat \varphi(\xi)|^2\ d\xi\,d\tau\notag\\
&\geq -\int_{\{|\xi|\leq 1\}}|\ln|\xi|^2| |\widehat \varphi(\xi)|^2\ d\xi\notag\\
&\geq-\delta\int_{\{|\xi|\leq 1\}}|\ln|\xi|^2| |\widehat \varphi(\xi)|^2\ d\xi
-(1-\delta)|\widehat \varphi|_\infty^2\int_{\{|\xi|\leq 1\}}|\ln|\xi|^2| \ d\xi\notag\\
&\geq\delta\int_{\{|\xi|\leq 1\}}\ln|\xi|^2 |\widehat \varphi(\xi)|^2\ d\xi
-C_2(1-\delta)|\varphi|_1^2
\label{it:4}
\end{align}
and, by \eqref{s0},
\begin{align}\label{it:5}
\int_0^1\int_{\{|\xi|>1\}}h_s(\tau)|\xi|^{2s\tau}\ln|\xi|^2 |\widehat \varphi(\xi)|^2\ d\xi\,d\tau
\geq \delta \int_{\{|\xi|>1\}}\ln|\xi|^2 |\widehat \varphi(\xi)|^2\ d\xi.
\end{align}
But then, by~\eqref{it:3},~\eqref{it:4}, and~\eqref{it:5},
\begin{align}\label{it:7}
\cE_L(\varphi,\varphi)=\int_{\R^N}\ln|\xi|^2 |\widehat \varphi(\xi)|^2\ d\xi\leq \frac{C_1}{\delta}
    +C_2\frac{1-\delta}{\delta}|\Omega||\varphi|_2^2.
\end{align}
By Lemma \ref{lem:sig:u}, there is $C_3=C_3(\Omega,C_0)>0$ such that 
\begin{align}\label{it:8}
|\varphi|_2^2<C_3.
\end{align}
 Moreover, by~\eqref{cEuu} and Lemma~\ref{lem:5.5},
 \begin{align}\label{it:9}
 \cE_L(\varphi,\varphi)\geq \|\varphi\|^2-C_4|\varphi|_2^2
 \end{align}
 for some $C_4=C_4(\Omega)>0$. The claim now follows, for $\varphi\in C^\infty_c(\Omega)\cap \cN_s$, from~\eqref{it:7},~\eqref{it:8}, and~\eqref{it:9}. The general case $\varphi\in\cN_s$ follows from the density of $C^\infty_c(\Omega)$ in $\cH^s_0(\Omega)$ and Lemma~\ref{lem:cEbd}.
\end{proof}

We are ready to show the main result in this section. 

\begin{proposition}\label{cor:Ebd}
Let $s\in(0,\tfrac{1}{4})$, $p_s=p(s)$ with $p$ as in~\eqref{p:hyp}, and let $u_{s}\in\cN_s$ be a least-energy solution of~\eqref{subcritical:intro}. There is $C=C(p,\Omega)>0$ such that 
\begin{align*}
\|u_s\|^2=\cE(u_{s},u_s)<C\qquad \text{ for all }s\in(0,\tfrac{1}{4}).
\end{align*}
\end{proposition}
\begin{proof}
Let $\varphi\in C^\infty_c(\Omega)\backslash \{0\}$.  By~\eqref{les} and Lemma~\ref{lem:t0},
\begin{align}\label{C0}
\|u_s\|_s^2=\inf_{v \in \cN_s}\|v\|_s^2\leq (t^s_\varphi)^2\|\varphi\|_s^2\leq \sup_{s\in(0,\frac{1}{4})}(t^s_\varphi)^2\|\varphi\|_s^2=:C_0<\infty.
\end{align}
With this bound, the result follows from Lemma~\ref{lem:vip}.
\end{proof}

\section{Existence of a least-energy solution of the limiting problem}\label{Sec:MT2}
In this section we show Theorem~\ref{main:thm:2}.  In the following, we use that, by \cite[Theorem 1.4]{CW19},
\begin{align}\label{thm:eigen_ll}
\lambda_1^L:=\min\{\cE_{L}(u,u): u\in\mathbb H(\Omega), \ |u|_2=1\}\in\R.    
\end{align}

Note that $\lambda_1^L$ can be zero or negative. In fact, by \cite[Corollary 1.10]{CW19}, $\lambda_1^L\leq \ln(\lambda_1)$, where $\lambda_1$ is the first Dirichlet eigenvalue of $(-\Delta)$ in $\Omega$. 

\begin{lemma}\label{lem:bound_below_nehari} There are $c_0>0$ and $c_1>0$ such that $|u|_2\geq c_0$ and $\|u\|>c_1$ for every $u\in\mathcal N_0(\Omega)$.
\end{lemma}
\begin{proof} 
Let $\mu=\frac{4}{N}\lambda$ for some $\lambda\in(0,1)$. For $u\in\mH(\Omega)$, let
$
F_0(u):=\cE_L(u,u)-\frac{2\lambda}{N}\int_{\Omega}u^2 \ln(u^2).
$
Then, by~\eqref{eq:ineq_becker},
\begin{equation}
F_0(u) \geq  (1-\lambda) \cE_L (u,u)-\left[2a_{N}+\frac{2}{N}\ln(|u|_2^2)\right]\lambda|u|_2^2 \qquad\textnormal{ for every } u\in\mathbb H(\Omega).
\end{equation}
Furthermore, using~\eqref{thm:eigen_ll},
\begin{equation}
F_0(u) \geq  \left[\frac{1-\lambda}{\lambda} \lambda_1^{L}-2a_{N}-\frac{2}{N}\ln(|u|_2^2)\right]\lambda|u|_2^2>0
\end{equation}
if $|u|_2 < \exp\left({\frac{(1-\lambda)}{4\lambda}N\lambda_1^L-\frac{N}{2}a_{N}}\right)=: c_0$. Therefore $|u|_2\geq c_0$ for every $u\in\mathcal N_0(\Omega)$. Finally, by~\eqref{comp:em}, there is $C>0$ such that $c_0\leq|u|_2\leq C\|u\|$ for all $u\in\mH(\Omega)$, and this ends the proof. 
\end{proof}

\begin{lemma}\label{lem:A1A2} For $w\in \mathbb H\setminus\{0\}$, let
\begin{align}\label{t0:def}
t^0_w:=\exp\left(\frac{\cE_L(w,w)-\mu\int_{\Omega}w^2\ln|w|}{\mu|w|_2^2}\right)    
\end{align}
and let $\alpha_w(s):=J_0(sw)$. Then, $\alpha_w^\prime(s)>0$ for $0<s<t^0_w$ and $\alpha_w^\prime(s)<0$ for $s>t^0_w$. In particular, $s\mapsto J_0(sw)$ achieves its unique maximum at $s=t^0_w$ and $t^0_w w\in\cN_0$.
\end{lemma}
\begin{proof} Note that 
\begin{align*}
\alpha'_w(s)
=\left(\cE_L(w,w)-\mu\int_\Omega w^2\ln|sw|\right)s
=\left(\cE_L(w,w)-\mu|w|_2^2\ln|s|-\mu\int_\Omega w^2\ln|w|\right)s
\end{align*}
for $w\in\mathbb H(\Omega)$. The claim now follows by the definition of $s_w$ and a direct computation.

\end{proof}

\begin{proposition}\label{prop:bd}
Let $(u_n)_{n\in\mathbb N}\subset \mathcal N_0$ be a sequence such that $\sup_{n\in\mathbb N} J_0(u_n)\leq C$ for some $C>0$. Then $(u_n)_{n\in\N}$ is bounded in $\mathbb H(\Omega)$ and, passing to a subsequence, there is $u\in\mathbb H(\Omega)\backslash\{0\}$ such that $u_n\weakly u$  weakly in $\mathbb H(\Omega)$ and $u_n\to u_0$ strongly in $L^2(\Omega)$ as $n\to\infty$.
\end{proposition}
\begin{proof}
Note that $J_0(u_n)=\frac{\mu}{4}|u_n|_2^2$, and therefore $\sup_{n\in\N}|u_n|_2^2\leq \frac{4}{\mu}C=:C_1$. Moreover, by Proposition~\ref{log:prop},
\begin{align*}
J_0(u_n)\geq \left(1-\frac{N\mu}{4}\right)\cE_L(u_n,u_n)-\frac{\mu}{2}\ln(|u_n|_2^2)|u_n|_2^2-a_N\frac{N\mu}{4}|u_n|^2_2,
\end{align*}
which yields that 
$\sup_{n\in\N}\cE_L(u_n,u_n)
\leq C+\sup_{t\in[0,C_1]}\left(\frac{\mu}{2}|\ln(t)|t+|a_N|\frac{N\mu}{4}t\right)=:C_2
$. By Lemma~\ref{lem:5.5},
\begin{align*}
C_2\geq \cE_L(u_n,u_n)\geq \|u_n\|^2 - (|\Omega|+|\rho_N|)C_1,
\end{align*}
which implies that $\sup_{n\in \N}\|u_n\|<\infty$.  Then, by~\eqref{comp:em}, passing to a subsequence, there is $u\in\mH(\Omega)$ such that $u_n\weakly u$ in $\mathbb H(\Omega)$ and $u_n\to u$ in $L^2(\Omega)$. Finally, by Lemma~\ref{lem:bound_below_nehari}, there is $c_0>0$ such that $|u_0|_2=\lim_{n\to\infty}|u_n|_2\geq c_0>0$ and therefore $u_0\neq 0$.
\end{proof}

\begin{lemma}\label{lem:L2conv}
Let $u_0\in L^2(\Omega)$ and let $(u_{k})_{k\in\N}\subset L^2(\Omega)$ be such that $u_k\to u_0$ in $L^2(\Omega)$ as $k\to\infty$. If $(\alpha_k)_{k\in\N}\subset[0,\tfrac{1}{2})$ is such that $\lim_{k\to\infty}\alpha_k=0$, then, passing to a subsequence,
\begin{align*}
\lim_{k\to\infty}\int_\Omega \ln(u_k^2)|u_k|^{\alpha_k}u_k\varphi \, dx
=\int_\Omega \ln(u_0^2)u_0\varphi \, dx\qquad \text{ for all }\varphi\in C^\infty_c(\Omega).
\end{align*}
\end{lemma}
\begin{proof}
Note that $\sup_{t\in(0,1),\,k\in\N}|t|^{\alpha_k+1}|\ln(t^2)|<\infty$ and, since $u_k\to u_0$ in $L^2(\Omega)$ as $k\to\infty$ by assumption, passing to a subsequence we have that $u_k\to u_0$ a.e. in $\Omega$ as $k\to\infty$.  Then, since $\Omega$ is bounded, we can use the dominated convergence theorem to obtain that
\begin{align}
\lim_{k\to\infty}\int_{\{|u_{k}|\leq 1\}} \ln (u_{k}^2) 
|u_{k}|^{\alpha_k}u_{k}\varphi\ dx
=\int_{\{|u_{0}|\leq 1\}} \ln (u_{0}^2)u_{0}\varphi\ dx.\label{bd0}
\end{align}
On the other hand, since $u_{k}\to u_0$ in $L^2(\Omega)$ as $k\to\infty$, there is a majorant $U\in L^2(\Omega)$ (see \emph{e.g.} \cite[Lemma A.1]{W96}) such that, passing to a subsequence, 
\begin{align*}
|u_{k}|<U\qquad \text{in $\Omega$ for all $k\in\N$.}     
\end{align*}
Moreover, on the set $\{|u_{k}|>1\}$, we can use Lemma~\ref{ln:bd} (with $\alpha=\frac{3}{2}<2=\beta$) to obtain that
\begin{align*}
\ln (u_{k}^2)|u_{k}|^{\alpha_k}u_{k}\varphi
\leq\ln (u_{k}^2)|u_{k}|^{\frac{3}{2}}|\varphi|\leq 4|u_{k}|^2|\varphi|<4|U|^2\|\varphi\|_\infty\qquad \text{in $\{|u_k|>1\}$ for all }k\in\N.
\end{align*}
By dominated convergence, then
\begin{align*}
\lim_{k\to\infty}\int_{\{|u_{k}|>1\}} \ln (u_{k}^2) 
|u_{k}|^{\alpha_k}u_{k}\varphi\ dx
=\int_{\{|u_{0}|>1\}} \ln (u_{0}^2)u_{0}\varphi\ dx,
\end{align*}
which, together with~\eqref{bd0}, yields the desired result.
\end{proof}

\begin{lemma}\label{mp:lem:ln}
It holds that
\begin{align}\label{mountain:pass:eq:lem}
 \inf_{\cN_0} J_0 = \inf_{\sigma\in \mathcal T}\max_{t\in[0,1]}J_0(\sigma(t)),
 \end{align}
 where $\mathcal T:=\{\sigma\in C^0([0,1],\mathbb H(\Omega)): \sigma(0)=0, \sigma(1)\neq 0, J_0(\sigma(1)\leq 0)\}$.
\end{lemma}
\begin{proof}
Let $v\in\cN_0$, then
\begin{align*}
J_0(tv)&=t^2\cE_L(v,v)-t^2\frac{\mu}{4}\int_\Omega v^2(\ln|(tv)^2|-1)\\
&=t^2\left(\cE_L(v,v)+\frac{\mu}{4}(1-2\ln|t|)|v|_2^2-\frac{\mu}{4}\int_\Omega v^2\ln(v^2)\right).
\end{align*}
Therefore, there is $r_v>0$ such that $J_0(r_vv)<0$, and setting $\sigma_v(t):=tr_vv$ we have that $\sigma_v\in \mathcal T$.  Note that $\max_{t\in[0,1]}J_0(\sigma_v(t))=J_0(v)$ (see Lemma~\ref{lem:A1A2}) and 
\begin{align}\label{mountain1}
\inf_{\sigma\in\mathcal T}\max_{t\in[0,1]}J_0(\sigma(t))\leq \inf_{v\in \cN_0}\max_{t\in[0,1]}J_0(\sigma_v(t))
=\inf_{v\in \cN_0}J_0(v).
\end{align}
On the other hand, let $\kappa:\mathbb H(\Omega)\to \R$ be given by 
\begin{align*}
    \kappa(v):=
    \begin{cases}
    \exp\left(\frac{\mu\int_\Omega v^2\ln|v|\, dx-\cE_L(v,v)}{|v|_2^2}\right), &\text{ if }v\neq0,\\
    0, &\text{ if }v=0.
    \end{cases}
\end{align*}
By Proposition~\ref{log:prop},
\begin{align*}
\frac{\mu\int_\Omega v^2\ln|v|\, dx-\cE_L(v,v)}{|v|_2^2}
\leq 
\mu\ln(|v|_2)+\frac{N\mu}{4} a_N,
\end{align*}
and therefore $\kappa$ is continuous at $v=0$.  Note that $\kappa(v)=1$ if and only if $v\in\cN_0$. Furthermore, if $v\neq 0$ and $J_0(v)\leq 0$, then $\kappa(v)>1$. But then, for every $\sigma\in\mathcal T$, $\kappa(\sigma(0))=0$, $\kappa(\sigma(1))>1$, and then there is $t_0\in(0,1)$ such that $\kappa(\sigma(t_0))=1$, which implies that $\sigma(t_0)\in\cN_0$. This yields that $\max_{t\in[0,1]}J_0(\sigma(t))\geq J_0(\sigma(t_0))\geq \inf_{\cN_0}J_0$; but then $\inf_{\sigma\in\mathcal T}\max_{t\in[0,1]}J_0(\sigma(t))\geq \inf_{\cN_0}J_0.$ This, together with~\eqref{mountain1}, implies~\eqref{mountain:pass:eq:lem} and ends the proof.
\end{proof}

We are ready to show Theorem~\ref{main:thm:2}.

\begin{proof}[Proof of Theorem~\ref{main:thm:2}]
Let $\Psi:\mH(\Omega)\backslash\{0\}\to\R$ be given by $\Psi(u):=\cE_L(u,u)-\frac{\mu}{2} \int_\Omega u^2 \ln(u^2) \ dx$. Then $\mathcal N_0=\Psi^{-1}(0)$ and 
$
\Psi'(u)u=2\cE_L(u,u)-\mu \int_\Omega (\ln(u^2)+1)u^2 \ dx=-\mu|u|_2^2<0$ if $u\in\cN_0$. In particular, $\cN_0$ is a $C^1$-manifold.  By Ekeland's variational principle \cite[Corollary 3.4]{E74}, there are $(u_n)_{n\in\N}\subset \cN_0$ and $(\zeta_n)_{n\in\N}\subset \R$ such that
\begin{align}
0\leq J_0(u_n)-\inf_{\cN_0} J_0\leq \frac{1}{n^2}
\quad \text{
and }\quad
\left\|
J_0'(u_n)-\zeta_n \Psi'(u_n)
\right\|_{\cL(\mH(\Omega),\R)}\leq \frac{1}{n}.
 \label{eke1}
\end{align}
In particular, $J_0(u_n)<\infty$ for all $n\in\N$ and
\begin{align}
o(1)&=\frac{1}{\|u_n\|}\left(J_0'(u_n)u_n-\zeta_n \Psi'(u_n)u_n\right)\notag\\
&=\frac{1}{\|u_n\|}\left(
\cE_L(u_n,u_n)-\frac{\mu}{2}\int_\Omega u_n^2 \ln(u_n^2)
+\zeta_n \mu|u_n|_2^2\right)
=\zeta_n \mu\frac{|u_n|_2^2}{\|u_n\|}\label{star}
\end{align}
as $n\to\infty$.  By Proposition~\ref{prop:bd}, there is $C_2>0$ satisfying that
\begin{align}\label{unbd}
\|u_n\|<C_2\qquad \text{ for all }n\in\N
\end{align}
and there is $u_0\in \mathbb H(\Omega)\backslash\{0\}$ such that, passing to a subsequence, $u_n\weakto{u_0}$ weakly in $\mathbb H(\Omega)$ and $u_n\to{u_0}$ strongly in $L^2(\Omega)$ as $n\to\infty$.  By \Cref{lem:bound_below_nehari}, there is $c_0>0$ such that $\frac{|u_n|_2^2}{\|u_n\|}>\frac{c_0}{C_2}$ for all $n\in\N$. Therefore~\eqref{star} implies that $\zeta_n\to 0$ as $n\to\infty.$ Moreover, by~\eqref{unbd}, there is $C_3>0$ such that 
\begin{align*}
|\Psi'(u_n)v| = \left|2\cE_L(u_n,v)-\mu \int_\Omega (\ln(u_n^2)+1)u_nv \ dx\right|<C_3
\end{align*}
 for all $v\in \mH(\Omega)$ with $\|v\|=1$, where we used that $I'(u_n)$ is a bounded linear operator, by Lemma~\ref{I:C1}.  As a consequence, by~\eqref{eke1},
\begin{align}\label{Jpz}
\left\|
J_0'(u_n)
\right\|_{\cL(\mH(\Omega),\R)}\leq \frac{1}{n}
+\zeta_n \|\Psi'(u_n)\|_{\cL(\mH(\Omega),\R)}\to 0\quad \text{ as }n\to\infty.
\end{align}
Then, by Lemma~\ref{lem:L2conv},
\begin{align*}
0
=
\lim_{n\to\infty}J_0'(u_n)\varphi=
\lim_{n\to\infty}\cE_L(u_n,\varphi)-\frac{\mu}{2}\int_\Omega \ln(u_n^2)u_n \varphi\ dx
=\cE_L(u_0,\varphi)-\frac{\mu}{2}\int_\Omega \ln(u_0^2)u_0\varphi\ dx
\end{align*}
for all $\varphi\in C^\infty_c(\Omega)$. This implies that $u_0$ is a weak solution of~\eqref{lambda:problem}. Moreover, since $C^\infty_c(\Omega)$ is dense in $\mH(\Omega)$, there is $(\varphi_n)_{n\in\N}\subset C^\infty_c(\Omega)$ such that $\varphi_n\to u_0$ in $\mH(\Omega)$ as $n\to\infty$. Then, by Lemma~\ref{I:C1}, we have that 
\begin{align*}
0=\lim_{n\to \infty}\cE_L(u_0,\varphi_n)-\frac{\mu}{2}\int_\Omega \ln(u_0^2)u_0\varphi_n\ dx
=\cE_L(u_0,u_0)-\frac{\mu}{2}\int_\Omega \ln(u_0^2)u_0^2\ dx.
\end{align*}
 Therefore, $u_0\in\cN_0$. Finally, using that $u_n,u_0\in\cN_0$ and that $u_n\to u_0$ in $L^2(\Omega),$
\begin{align*}
\inf_{\cN_0}J_0&=\lim_{n\to\infty}J_0(u_n)
=\lim_{n\to\infty}\cE_L(u_n,u_n)-\frac{\mu}{4}\int_\Omega u_n^2(\ln(u_n^2)-1)\ dx\\
&=\lim_{n\to\infty}\frac{\mu}{4}\int_\Omega u_n^2\ dx=
\frac{\mu}{4}\int_\Omega u_0^2\ dx=J_0(u_0).
\end{align*}
Note that~\eqref{mountain:pass} follows from Lemma \ref{mp:lem:ln}.

Finally, let $u_0$ be any least-energy solution and we argue that $u_0$ does not change sign. By \cite[Lemma 3.3]{CW19}, we have that $|u_0|\in\mathbb H(\Omega)$ and 
\begin{align}\label{prop:comp_EL}
    \cE_{L}(|u_0|,|u_0|)\leq \cE_{L}(u_0,u_0).
\end{align}
Furthermore, the equality holds if and only if $u_0$ does not change sign. Let $t^0_{|u_0|}$ be given by~\eqref{t0:def} with $w=|u_0|$. Then $t_{|u_0|}^0|u_0|\in\cN_0$ and, by~\eqref{prop:comp_EL} and because $u_0\in\cN_0$, we have that $t_{|u_0|}^0\leq 1$. Therefore
\begin{align*}
J_0(u_0)=\inf_{\cN_0}J_0\leq J_0(t_{|u_0|}^0|u_0|)=\frac{\mu}{4}(t_{|u_0|}^0)^2|u_0|_2^2\leq \frac{\mu}{4}|u_0|_2^2=J_0(u_0).
\end{align*}
This yields that $t_{|u_0|}^0=1$ and therefore~\eqref{prop:comp_EL} must hold with equality.  This implies that $u_0$ does not change sign.

\end{proof}

\section{Convergence of solutions}\label{6:sec}

In this section we show that least-energy solutions of~\eqref{subcritical:intro} converge in the $L^2$-sense, up to a subsequence, to a least-energy solution of~\eqref{log:prob:intro}.  First, a standard use of the Nehari method yields the existence of least-energy solutions of~\eqref{subcritical:intro}.  For completeness (and for comparison with Theorem~\ref{main:thm:2}) we include a short proof. 

\begin{theorem}\label{existence}
Let $s\in(0,\tfrac{1}{4})$ and $p_s\in(2,2^*_s)$. There is a least-energy solution $u_s\in\cN_s$ of~\eqref{subcritical:intro}, namely, $u_s$ satisfies~\eqref{ws:intro} and $J_s(u_s)=\inf_{\cN_s}J_s.$ Furthermore, all least-energy solutions of~\eqref{subcritical:intro} are either positive or negative in $\Omega$.
\end{theorem}
\begin{proof}
Let $\Psi:\cH^s_0(\Omega)\backslash\{0\}\to\R$ be given by $\Psi(u):=\|u\|_s^2-|u|_{p_s}^{p_s}$. Then $\mathcal N_s=\Psi^{-1}(0)$ and 
$
\Psi'(u)u=2\|u_s\|^2_s-p_s|u_s|_{p_s}^{p_s}=(2-p_s)\|u_s\|_s^2<0$ if $u\in\cN_s$. By Ekeland's variational principle \cite[Corollary 3.4]{E74} and Lemma~\ref{lem:Ebd}, there are $(u_n)_{n\in\N}\subset \cN_s$, $(\zeta_n)_{n\in\N}\subset \R$, and $C>1$ such that
\begin{align}
C^{-1}\leq \|u_n\|_s\leq C,\qquad
0\leq J_s(u_n)-\inf_{\cN_s} J_s\leq \frac{1}{n^2},
\qquad
\left\|
J_s'(u_n)-\zeta_n \Psi'(u_n)
\right\|_{\cL(\cH_0^s(\Omega),\R)}\leq \frac{1}{n}
 \label{eke1:s}
\end{align}
for all $n\in\N$. It follows that, $o(1)=\frac{1}{\|u_n\|_s}\left(J_s'(u_n)u_n-\zeta_n \Psi'(u_n)u_n\right)
=\zeta_n (2-p_s)\|u_n\|_s$ as $n\to\infty$ and therefore $\zeta_n\to 0$ as $n\to\infty.$ Then $\left\|J_s'(u_n)\right\|_{\cL(\cH^s_0(\Omega),\R)}\to 0$ as $n\to\infty$ and there is $u_s\in \cH_0^s(\Omega)\backslash\{0\}$ such that, passing to a subsequence, $u_n\weakto{u_s}$ weakly in $\cH^s_0(\Omega)$ and $u_n\to{u_s}$ strongly in $L^{p_s}(\Omega)$ as $n\to\infty$. With these facts, we conclude that $u_s$ is a weak solution of~\eqref{subcritical:intro}, $u_s\in\cN_s$, and, up to a subsequence,
\begin{align*}
\inf_{\cN_s}J_s&=\lim_{n\to\infty}J_s(u_n)
=\left(\frac{1}{2}-\frac{1}{p_s}\right)\lim_{n\to\infty}\|u_n\|_s^2\geq \left(\frac{1}{2}-\frac{1}{p_s}\right)\|u_s\|_s^2= J_s(u_s)\geq \inf_{\cN_s}J_s.
\end{align*}
Finally, let $u_s$ be a least-energy solution and let $t_{|u_s|}^s$ be given by~\eqref{ts:def}. Then $\||u_s|\|_s\leq \|u_s\|_s$, $t_{|u_s|}^s\leq 1$, and $t_{|u_s|}^s|u_s|\in\cN_s$. Therefore,
\begin{align*}
J_s(u_s)\leq J_s(t_{|u_s|}^s|u_s|)=\left(\frac{1}{2}-\frac{1}{p_s}\right)(t_{|u_s|}^s)^{p_s}|u_s|_{p_s}^{p_s}\leq J_s(u_s),
\end{align*}
which yields that $t_{|u_s|}^s=1$ and $|u_s|$ is a nonnegative least-energy solution of~\eqref{subcritical:intro}.  By the strong maximum principle (see \emph{e.g.} \cite[Proposition 3.3]{JF}) it follows that $|u_s|>0$ in $\Omega$ and therefore $u_s$ is either strictly positive or negative in $\Omega$.
\end{proof}

Next we recall some known properties. Let $u\in\mathbb H(\Omega)$ and $\varphi\in C^\infty_c(\Omega)$, then
\begin{align}\label{ibyp}
    \cE_L(u,\varphi)=\int_\Omega u L_\Delta \varphi\ dx,
\end{align}
see \cite[eq. (3.11)]{CW19}. Moreover, by \cite[Theorem 1.1]{CW19},  $L_\Delta \varphi\in L^p(\R^N)$ and
\begin{align}\label{exp:thm}
    \lim_{s\to 0^+}\left|\frac{(-\Delta)^s \varphi- \varphi}{s}-L_\Delta \varphi\right|_p=0\qquad \text{ for all }0<p\leq \infty.
\end{align}
In particular, $(-\Delta)^s \varphi=\varphi+s L_\Delta \varphi+o(s)$ in $L^\infty(\R^N)$ as $s\to 0^+$.

\medskip

We are ready to show our main theorem.

\begin{proof}[Proof of Theorem~\ref{main:thm}]
Let $s_k\in(0,\frac{1}{4})$ be such that $s_k\to 0$ and let $(u_{s_k})_{k\in \N}$ be a sequence of least-energy solutions (the set of least-energy solutions is nonempty by Theorem~\ref{existence}). By~\eqref{C0} there is $C_0=C_0(\Omega,p)>0$ such that
\begin{align}\label{C0thm}
\|u_{s_k}\|_{s_k}<C_0\qquad \text{ for all }k\in\N.
\end{align}
Moreover, by Proposition~\ref{cor:Ebd}, the sequence $(u_{s_k})_{k\in \N}\subset \cN_{s_k}$ is uniformly bounded in the Hilbert space $\mathbb H(\Omega)$.  By the compact embedding of $\mathbb H(\Omega)$ into $L^2(\Omega)$, we obtain that, passing to a subsequence, there is $u_0\in \mathbb H(\Omega)$ such that 
\begin{align}\label{conv}
    u_{s_k}\rightharpoonup u_0\quad \text{ in }\mathbb H(\Omega),\qquad
    u_{s_k}\to u_0\quad \text{ in }L^2(\Omega)\text{ as }k\to\infty.
\end{align}

Observe that, if $f(s):=|t|^{p_s-2}t$, then 
$f'(s\tau)=p'(s\tau)\ln(|t|)|t|^{p(s\tau)-2}t$. Let $\varphi\in C^\infty_c(\Omega)$, then, by~\eqref{exp:thm},
\begin{align}
\int_\Omega u_{s_k} (\varphi+{s_k} L_\Delta \varphi + o({s_k}))
&=\int_\Omega u_{s_k} (-\Delta)^{s_k}\varphi= \int_\Omega |u_{s_k}|^{p_{s_k}-2}u_{s_k}\varphi\notag\\
&= \int_\Omega \left(u_{s_k}+{s_k}\int_0^1
p'(s_k\tau)\ln(|u_{s_k}|)|u_{s_k}|^{p(s_k\tau)-2}u_{s_k}
\ d\tau\right)\varphi\ dx,\label{a:eq}
\end{align}
in $L^\infty(\Omega)$ as $k\to \infty$.
Moreover, since $\int_\Omega u_{s_k} L_\Delta \varphi = \cE_L (u_{s_k},\varphi)$, by~\eqref{ibyp}, then~\eqref{a:eq} implies that
\begin{align}
\cE_L(u_{s_k},\varphi)+o(1)
&=\int_\Omega u_{s_k} L_\Delta \varphi\ dx + o(1)\notag\\
&=\int_\Omega\int_0^1
p'(s_k\tau)\ln(|u_{s_k}|)|u_{s_k}|^{p(s_k\tau)-2}u_{s_k}
\ d\tau\varphi\ dx,\label{rhs}
\end{align}
as $k\to \infty$ for all $\varphi\in C^\infty_c(\Omega)$. By Lemma~\ref{lem:L2conv}, passing to a subsequence,
\begin{align*}
\lim_{k\to\infty}\int_0^1 p'(s_k\tau) \int_{\Omega} \ln (|u_{s_k}|) 
|u_{s_k}|^{p(s_k\tau)-2}u_{s_k}
\varphi\ dx\ d\tau
=p'(0)\int_{\Omega} \ln (|u_{0}|)u_{0}\varphi\ dx.
\end{align*}
Therefore, letting $k\to \infty$ in~\eqref{rhs} we conclude that
\begin{align}\label{ws}
\cE_L(u_0,\varphi)
=p'(0) \int_\Omega \ln (|u_0|)u_0\varphi\qquad \text{ for all }\varphi\in C^\infty_c(\Omega).
\end{align}
Then, by density, $u_0$ is a weak solution of 
\begin{align*}
 L_\Delta u_0 = p'(0) \ln (|u_0|)u_0\quad \text{ in }\Omega.
\end{align*}

Let
\begin{align*}
    \lambda_k=\frac{p(s_k)-2}{2^*_s-2}\in(0,1),\quad
\alpha_k=(1-\lambda_k)2,\quad
\beta_k=\lambda_k 2^*_{s_k},\quad
r_k=\frac{1}{(1-\lambda_k)},\quad
q_k=\frac{1}{\lambda_k},
\end{align*}
 and note that $\alpha_k+\beta_k=p(s_k)$, $\frac{1}{r_k}+\frac{1}{q_k}=1,$ and 
 \begin{align*}
  \lim_{k\to\infty} \lambda_k 
  = \frac{s_k\int_0^1 p'(s_k\tau)\ d\tau}{s_k \frac{4}{N-2s_k}}
  = p'(0)\frac{N}{4}\in(0,1),
 \end{align*}
 by~\eqref{p:hyp}. 
 Then, by~\eqref{bds}, Theorem~\ref{thm:sobolev},~\eqref{C0thm}, and Hölder's inequality, there is $c=c(\Omega,p)>0$ and $C=C(\Omega,p)>0$ such that
\begin{align*}
c&<
\|u_{s_k}\|_{s_k}^2
=\int_\Omega|u_{s_k}|^{p_{s_k}}
=\int_\Omega|u_{s_k}|^{\alpha_k}|u_s|^{\beta_k}\\
&\leq \left(\int_\Omega|u_{s_k}|^{\alpha_k r_k}\right)^\frac{1}{r_k} \left(\int_\Omega|u_{s_k}|^{\beta_k q_k}\right)^\frac{1}{q_k}=|u_{s_k}|_2^{2(1-\lambda_k)} |u_{s_k}|_{2^*_{s_k}}^{2^*_{s_k}\lambda_k}
\leq C|u_{s_k}|_2^{2(1-\lambda_k)},
\end{align*}
for all $k\in\N$, and therefore 
\begin{align*}
|u_0|_2=\lim_{k\to \infty}|u_{s_k}|_{2}\geq \left(\frac{c}{C}\right)^{\frac{1}{2(1-\frac{N}{4}p'(0))}}>0,
\end{align*}
This yields that $u_0\neq 0$ is a nontrivial weak solution and $u_0\in\cN_0$. 

Next, we show that $u_0$ is a least-energy solution of the limiting problem, namely, that
\begin{align*}
\frac{p'(0)}{4}|u_0|_2^2=J_{0}(u_0)=\inf_{\cN_0}J_{0}=:c_0>0,
\end{align*}
where $J_0(u)=\frac{1}{2}\cE_L(u,u)-\frac{p'(0)}{4}\int_\Omega |u|^2(\ln(|u|^2)-1)\ dx$.  Noting that $\lim_{k\to\infty}\frac{1}{s_k}\left(\frac{1}{2}-\frac{1}{p_{s_k}}\right)=\frac{p'(0)}{4}$, we have that
\begin{align*}
J_{s_k}(u_{s_k})=\left(\frac{1}{2}-\frac{1}{p_{s_k}}\right)\|u_{s_k}\|_{s_k}^2\quad \text{ and }\quad 
\lim_{k\to\infty}\frac{1}{s_k}J_{s_k}(u_{s_k})=\frac{p'(0)}{4}\lim_{k\to\infty}\|u_{s_k}\|_{s_k}^2.
\end{align*}
Let $c_k:=\frac{1}{s_k}\left(\frac{1}{2}-\frac{1}{p_{s_k}}\right)\|u_{s_k}\|_{s_k}^2$. 
 Then, by~\eqref{C0thm}, there is $c^*\in\R$ such that, passing to a subsequence, $\lim_{k\to\infty}c_k=c^*$.  We claim that $c^*=c_0$.

By Fatou's Lemma,
\begin{align}\label{3}
    c_0=\frac{p'(0)}{4}|u_0|_2^2 \leq \frac{p'(0)}{4}\liminf_{k\to\infty}\int_{\R^N}|\xi|^{2s_k}|\widehat u_{s_k}|^2\ d\xi 
    =\liminf_{k\to\infty}c_k=c^*.
\end{align}

On the other hand, by Theorem~\ref{main:thm:2} with $\mu=p'(0)$, there is $v\in \cN_0$ such that $J_0(v)=c_0$. Let $(v_n)_{n\in \N}\subset C^\infty_c(\Omega)\cap \cN_0$ such that $v_n\to v$ as $n\to\infty$ in $\mH(\Omega)$. By Lemma~\ref{lem:t0} and using that $v_n\in\cN_0$,
\begin{equation}\label{1}
\lim_{k\to\infty}t^n_k= 1\quad \text{ for every }n\in\N,\text{ where }t^n_k:=\left(\frac{\|v_n\|_{s_k}^2}{|v_n|_{p_{s_k}}^{p_{s_k}}}\right)^{\frac{1}{p_{s_k}-2}}.
\end{equation}
But then, using the minimality of $u_{s_k}$,~\eqref{1}, Lemma~\ref{Es:exp:lem}, and that $t^n_k v_n\in\cN_s$,
\begin{align*}
c^*=\lim_{k\to\infty}c_{k}
=\lim_{k\to\infty}\frac{1}{s_k}J_{s_k}(u_{s_k})
\leq \lim_{k\to\infty}\frac{1}{s_k}J_{s_k}(t^n_{k}v_n)=\lim_{k\to\infty}\frac{1}{s_k}\left(\frac{1}{2}-\frac{1}{p_{s_k}}\right)\|t_{k}^nv_n\|_{s_k}^2
=\frac{p'(0)}{4}|v_n|_2^2.
\end{align*}
Since $\lim_{n\to\infty}\frac{p'(0)}{4}|v_n|_2^2=\frac{p'(0)}{4}|v|_2^2=J_0(v)=c_0$, we have that $c^*\leq c_0$. Together with~\eqref{3}, we conclude that $c^*=c_0$.  Finally, arguing as in~\eqref{3} and using~\eqref{conv},
\begin{align*}
c_0\leq J_0(u_0)=\frac{p'(0)}{4}|u_0|_2^2 \leq \frac{p'(0)}{4}\liminf_{k\to\infty}\|u_{s_k}\|_{s_k}^2
 =\lim_{k\to\infty}c_k=c^*=c_0,
\end{align*}
which implies that $J_0(u_0)=c_0$.  
\end{proof}

\section{Closing remarks}\label{c:rmks}

To finish this paper, we comment on the following.

\begin{remark}\label{open:rmk}(On the extremal cases for $p'(0)$) The cases $p'(0)=0$ and $p'(0)=\frac{4}{N}$ are not covered by Theorem~\ref{main:thm} (note that the assumption $2<p(s)<2^*_s$ for $s\in(0,\frac{1}{4})$ implies that $p'(0)\in[0,\frac{4}{N}]$).

For $p'(0)=0$, the  characterization of the limiting problem (see the proof of Theorem \ref{main:thm}) requires a second order \textemdash or even higher, if $p''(0)=0$\textemdash expansion of the fractional Laplacian and of the power nonlinearity at $s=0$, and we do not pursue this here. 

On the other hand, if $p'(0)=\frac{4}{N}$ (which corresponds to the ``critical case"), then we cannot use the logarithmic Sobolev inequality to obtain estimates in the $\mH(\Omega)$-norm (and gain compactness), see for example Lemma~\ref{lem:bound_below_nehari} and Proposition~\ref{prop:bd}, where the fact that $\lambda:=p'(0)\frac{N}{4}<1$ is crucial. For problem \eqref{subcritical:intro} with the critical Sobolev exponent $p_s=2^*_s$, the use of a suitable symmetric variational framework can yield the existence of solutions in some bounded (and unbounded) domains, see \emph{e.g.} \cite{HSS21} and the references therein.  We conjecture that a similar approach can be used to study \eqref{lambda:problem} with $\mu=\frac{4}{N}$.
\end{remark}

\begin{remark}(On the positivity properties of solutions) In Theorem \ref{existence}, a strong maximum principle for supersolutions of $(-\Delta)^s$ is used to show that least-energy solutions of \eqref{subcritical:intro} are either strictly positive or strictly negative in $\Omega$; in particular, for $u_s\in\cH^s_0(\Omega)$,
\begin{align*}
 (-\Delta)^s u_s = u_s^{p_s-1}\geq 0\quad \text{ in }\Omega  \quad \text{implies that} \quad u_s>0\quad \text{ in }\Omega. 
\end{align*}
Strong maximum principles for $L_\Delta$ are available (see for example the proof of \cite[Theorem~3.4]{CW19} where a general strong maximum principle for nonlocal operators shown in \cite[Theorem~1.1]{JW} is used to prove that the first Dirichlet eigenfunction of $L_\Delta$ is strictly positive). However, note that the logarithmic nonlinearity $\ln(u_0)u_0$ does not have a fixed sign for a nonnegative $u_0$, namely, for $u\in\mH(\Omega)$,
\begin{align*}
 L_\Delta u_0 = \mu \ln(u_0)u_0\quad\text{ in }\Omega\quad \text{ and }\quad u_0\geq 0\quad \text{ in }\Omega  \quad \text{does not imply that} \quad u_0>0\quad \text{ in }\Omega. 
\end{align*}
As a consequence, in Theorem \ref{main:thm:2}, we only show that the least-energy solutions are either nonnegative or nonpositive.
\end{remark}

\begin{remark}[Generalizations and extensions]\label{ext:rmk} Our main variational tool is the Nehari manifold method.  This is a very flexible and versatile approach that can be applied to study a wide set of nonlinear problems. For instance, in \cite{SW10}, a generalized Nehari method is used to show the existence of a least-energy solution (\emph{a.k.a.} ground state) of the problem
\begin{align}\label{L}
    -\Delta u -\lambda u= f(x,u)\quad \text{ in }\Omega,\qquad u=0\quad \text{ on }\partial \Omega,
\end{align}
where $\lambda\in\R$ and $f\in C(\Omega\times \R,\R)$ satisfies the following assumptions:
\begin{enumerate}
    \item (Subcriticality) $|f(x,u)|\leq a(1+|u|^{q-1})$ for some $a>0$ and $2<q<\frac{2N}{N-2}=2^*_1$, $N\geq 3$.
    \item (Superlinearity) $f(x,u)=o(u)$ uniformly in $x$ as $u\to 0$ and $F(x,u)/u^2\to\infty$ uniformly in $x$ as $|u|\to\infty$, where $F(x,u):=\int_0^u f(x,s)\ ds$.
    \item (Monotonicity) $u\mapsto f(x,u)/|u|$ is strictly increasing on $(-\infty,0)$ and $(0,\infty)$.
\end{enumerate}
    Note that $f$ is not assumed to be $C^1$.  Furthermore, if $f$ is odd in $u$, then \eqref{L} has infinitely many solutions. Combining the approach from \cite{SW10} with the methods in this paper, it is possible to characterize the small order asymptotics of \eqref{L} when $-\Delta$ is substituted with $(-\Delta)^s$. In this case, the condition for subcriticality would be $|f_s(x,u)|\leq a(1+|u|^{q_s-1})$ for some $a>0$ and $2<q_s<\frac{2N}{N-2s}=2^*_s$, $s\in(0,\frac{N}{2}),$ $N\geq 1$. To analyze the small order limit as $s\to 0^+$, suitable (logarithmic-subcriticality) assumptions need to be imposed on the behavior of the map $s\mapsto f_s$ at $s=0$, as in \eqref{p:hyp}.
\end{remark}

\renewcommand{\abstractname}{Acknowledgements}
\begin{abstract}
\end{abstract}
\vspace{-0.5cm}

Víctor Hernández-Santamaría is supported by the program \emph{Estancias posdoctorales por M\'exico} of CONACyT, Mexico.  Alberto Saldaña is supported by UNAM-DGAPA-PAPIIT grant IA101721, Mexico.  The authors thank Pierre Aime Feulefack, Sven Jarohs, and Tobias Weth for helpful discussions and Harbir Antil for sharing some relevant references. We also thank the anonymous referee for helpful comments and suggestions.

\end{document}